\documentclass[leqno,12pt]{article}
\usepackage{amsmath,amssymb,amsthm,graphicx,color}
\usepackage{epsfig}
\usepackage[colorinlistoftodos]{todonotes}
\usepackage{hyperref}
\usepackage[normalem]{ulem}

\def\LHS{left-hand side}
\def\RHS{right-hand side}

\newtheorem{theorem}{Theorem}[section]

\newtheorem{corollary}[theorem]{Corollary}

\newtheorem{definition}[theorem]{Definition}
\newtheorem{example}[theorem]{Example}

\newtheorem{proposition}[theorem]{Proposition}
\newtheorem{remark}[theorem]{Remark}

\def\dis{\displaystyle}

\def\N{\mathbb{N}}
\def\R{\mathbb{R}}

\def\sur{{\rm sur}\,}

\def\xx{\bar x}

\def\dist{\,{\rm dist}\,}

\def\xx{\bar x}
\def\bx{\bar x}
\def\by{\bar y}
\def\bz{\bar z}

\def\dom{\mathop{\rm dom}\nolimits} 
\def\gph{\mathop{\rm gph}\nolimits} 
\def\tto{\rightrightarrows}
\def\ball{{I\kern -.35em B}}
\def\bx{\bar x}

\def\by{\bar y}   

\def\reg{\mathop{\rm reg}\nolimits}
\def\lip{\mathop{\rm lip}\nolimits}
\def\lopen{\mathop{\rm lopen}\nolimits}
\def\subreg{\mathop{\rm subreg}\nolimits}
\def\semireg{\mathop{\rm semireg}\nolimits}
\def\psopen{\mathop{\rm popen}\nolimits}
\def\calm{\mathop{\rm calm}\nolimits}

\def\inT{\mathop{\rm int}\nolimits}

\def\dis{\mathop{\rm dist}\nolimits}

\begin{document}
\bigskip

\centerline{{\large \bf  On semiregularity of mappings}}

\bigskip

\bigskip

\centerline{ \bf
  R. Cibulka\footnote{NTIS - New Technologies for the Information Society and Department of Mathematics, Faculty of Applied Sciences, University of West Bohemia, Univerzitn\'\i\ 22, 306 14
                      Pilsen, Czech Republic, cibi@kma.zcu.cz. Supported by the project GA15-00735S.},
  M. Fabian\footnote{Mathematical Institute of Czech Academy of Sciences, \v Zitn\'a 25, 115 67 Praha 1, Czech Republic, fabian@math.cas.cz. Supported by
                     the project GA\v CR  17-00941S and by RVO: 67985840.},  
  A. Y. Kruger\footnote{Centre for Informatics and Applied Optimization, Federation University Australia, POB 663, Ballarat, VIC 3350, Australia, a.kruger@federation.edu.au. Supported by the
  Australian Research Council, project DP160100854.}
 }

\bigskip

\centerline{\today}

\bigskip

\begin{small}
\begin{quote}
{\bf Abstract.}
There are two basic ways of weakening the definition of the well-known metric regularity
property by fixing one of the points involved in the definition. The first
resulting property is called metric subregularity and has attracted a lot of attention during
the last decades. On the other hand,  the latter property which we call semiregularity can be found under
several names and the
corresponding results  are scattered in the literature. We provide a self-contained material gathering and extending
the existing theory on the topic. We demonstrate a clear relationship with other regularity properties, for example, the equivalence with the so-called openness with a linear rate at the reference point is shown. In particular cases, we derive necessary and/or sufficient conditions of both primal and dual type.   We illustrate the importance of semiregularity in the convergence analysis of  an inexact Newton-type scheme for generalized equations with not necessarily differentiable single-valued part.

\noindent
{\bf Key Words.}   open mapping theorem, linear openness, metric semiregularity, set-valued perturbation

\bigskip

\noindent
{\bf AMS Subject Classification (2010)} 49J53, 49J52, 49K40, 90C31.

\end{quote}

\bigskip

\bigskip

\end{small}

\section{Introduction}
The concept of \emph{regularity} of a set-valued mapping $F$ acting from a metric space $(X,d)$ into (subsets of) another metric space $(Y, \varrho)$, denoted by ${F:X \tto Y}$,  around a given reference point  $(\bx,\by)$ in its graph $\gph F$ plays a fundamental role in modern variational analysis and non-smooth optimization, see, for example, a recent survey  \cite{IoffeSurvey} by Ioffe or books \cite{AF,book, KKbook, JPP}. By regularity we mean that one of the three equivalent properties -- metric regularity, openness with a linear rate around the reference point, and pseudo-Lipschitz property\footnote{Often also called Lipschitz-like or Aubin property.}  of the inverse $F^{-1}$ -- holds for the mapping under consideration. First, the mapping $F$
is said to be  {\em   metrically regular}\footnote{In \cite{book}, this property is called metric regularity at $\bx$ for $\by$ under an additional assumption that the graph of $F$ is locally closed at the reference point.}   around $(\bx,\by)$ when $\by \in F(\bx)$ and  there is a constant $\kappa > 0$ along  with
a neighborhood $U \times V$ of $(\bx, \by)$ in $X \times Y$ such that
\begin{equation}
\label{defMR}
      \dis \big(x, F^{-1}(y) \big) \leq \kappa \dis\big(y, F(x)\big)
      \quad \mbox{for every} \quad  (x,y) \in U \times V,
\end{equation}
where $\dis(u,C)$ is the distance from a point $u$ to a set $C$ and the space $X \times Y$ is equipped with the product (box) topology. The infimum of $\kappa > 0$ for which there exists a neighborhood $ U \times V$ of $(\bx,\by)$ in $X \times Y$ such that \eqref{defMR} holds is called the {\it regularity modulus} of $F$
around $(\bx,\by)$ and is denoted by $\reg F (\bx,\by)$.

Second, the mapping $F$ is called {\it open with a linear rate}\footnote{There are other equivalent definitions in the literature. Also note that in \cite{book} the constant $c$ appears on the right-hand side of \eqref{defSur}.} around
$(\bar{x},\bar{y})$ when $\by \in F(\bx)$ and there are positive constants $c$ and $\varepsilon$ along  with
a neighborhood $ U \times V$ of $(\bx, \by)$ in $X \times Y$  such that
\begin{equation}\label{defSur}
\ball[y,ct] \subset F(\ball[x,t])
\quad  \mbox{whenever} \quad  (x,y) \in  U \times V, \quad y \in F(x) \quad \mbox{and} \quad  t  \in (0,\varepsilon),
\end{equation}
where  $\ball[u,r]$ denotes the closed ball centered at $u$ with a radius $r > 0$.
The  supremum of $c > 0$  for which there exist a constant $\varepsilon>0$ and a neighborhood $ U \times V$ of $(\bx, \by)$ in $X \times Y$ such that (\ref{defSur}) holds is called the
{\it modulus of surjection}
of $F$  around  $(\bar{x},\bar{y})$ and is denoted by
$\sur F(\bar{x},\bar{y})$ \footnote{Clearly, we can replace the closed balls in \eqref{defSur} with the open ones.}.
Finally, the mapping $F: X \tto Y$ is said to be {\it pseudo-Lipschitz} around $(\bx,\by)$ when $\by \in F(\bx)$ and there is a constant $\mu > 0$ along with a neighborhood $U \times V$ of $(\bx, \by)$ in $X \times Y$  such that
\begin{equation}\label{defLip}
 \dis\big(y, F(x)\big) \leq \mu \, d(x, x') \quad  \mbox{whenever} \quad x,x' \in U \quad \mbox{and} \quad y \in F(x') \cap V.
\end{equation}
The infimum of $\mu > 0$ for which there exists a neighborhood $U \times V$ of $(\bx,\by)$ in $X \times Y$ such that \eqref{defLip} holds is called the {\it Lipschitz  modulus} of $F$
 around $(\bx, \by)$ and is denoted by $\lip F (\bx,\by)$.

A fundamental well-known fact is that
\begin{equation} \label{eqBasic}
 \sur F(\bar{x},\bar{y}) \cdot \reg F (\bx,\by) = 1 \quad \mbox{and} \quad \reg F (\bx,\by) = \lip F^{-1} (\by,\bx),
\end{equation}
under the
convention that  $0 \cdot \infty =  \infty \cdot 0 = 1$,  $\inf \emptyset = \infty$, and, as we work with nonnegative quantities, that $\sup \emptyset = 0$.

Fixing one of the components of $(x,y)$ in \eqref{defMR}, that is letting either $x:=\bx$ or $y:=\by$, one gets two different, weaker than regularity, concepts. Of course, one can reformulate both of them in terms of openness and continuity of the inverse, respectively.
\begin{definition} \label{defSubreg}
Consider a mapping  $F:X \tto Y$ between metric spaces $(X, d)$ and $(Y,\varrho)$ and a~point $(\bx, \by) \in X \times Y$.
\begin{itemize}
 \item[(A1)] $F$ is said to be {\em metrically subregular at $(\bx,\by)$} when $\by \in F(\bx)$ and  there is a constant $\kappa > 0$ along  with
a neighborhood $U$ of $\bx$ in $X$ such that
\begin{equation}
\label{defMSR}
      \dis \big(x, F^{-1}(\by) \big) \leq \kappa \dis\big(\by, F(x)\big)
      \quad \mbox{for every} \quad  x \in U.
\end{equation}
 The infimum of $\kappa > 0$ for which there exists a neighborhood $U$ of $\bx$ in $X$ such that \eqref{defMSR} holds is called the {\em subregularity modulus} of $F$
 at $(\bx,\by)$ and is denoted by $\subreg F (\bx,\by)$;
  \item[(A2)] $F$ is said to be  {\em pseudo-open with a linear rate at $(\bar{x},\bar{y})$} when $\by \in F(\bx)$ and there are positive constants $c$ and $\varepsilon$  along  with
a neighborhood $U$ of $\bx$ in $X$  such that
\begin{equation}\label{defSSur}
\by  \in  F(\ball[x,t]) \quad  \mbox{whenever} \quad  x \in U \cap F^{-1}(\ball[\by, ct])   \quad \mbox{and} \quad  t  \in (0,\varepsilon).
\end{equation}
The  supremum of $c > 0$ for which there exist a constant $\varepsilon>0$ and a neighborhood $U$ of $\bx$ in $X$ such that (\ref{defSSur}) holds   is called the
{\em modulus of pseudo-openness} of $F$  at  $(\bar{x},\bar{y})$ and is denoted by
$\psopen F(\bar{x},\bar{y})$;
\item[(A3)]  $F$ is said to be {\em calm at $(\bx,\by)$} when $\by \in F(\bx)$ and there is a constant $\mu > 0$ along with a~neighborhood $U\times V$ of $(\bx, \by)$ in $X\times Y$  such that
\begin{equation}\label{defCalm}
 \dis\big(y, F(\bx)\big) \leq \mu \, d(\bx,x) \quad  \mbox{whenever} \quad x \in U \quad \mbox{and} \quad y \in F(x)\cap V.
\end{equation}
The infimum of $\mu > 0$ for which there exists a neighborhood $U\times V$ of $(\bx,\by)$ in $X\times Y$ such that \eqref{defCalm} holds is called the {\em calmness  modulus} of $F$
at $(\bx, \by)$ and is denoted by $\calm F (\bx,\by)$.
\end{itemize}
\end{definition}

Properties in (A1) and (A3) are entrenched in the literature \cite{JPP, book} and the metric subregularity of a mapping is known to be equivalent to the calmness of its inverse. (A2) is defined and proved to be equivalent with the remaining ones in \cite{ADS2013}. More precisely, the following analogue of \eqref{eqBasic} holds true
\begin{equation} \label{eqBasicA}
 \psopen F(\bar{x},\bar{y}) \cdot \subreg F (\bx,\by) = 1 \quad \mbox{and} \quad \subreg F (\bx,\by) = \calm F^{-1} (\by,\bx).
\end{equation}
The case when $x:=\bx$ in \eqref{defMR}, being the same as letting $(x,y):=(\bx,\by)$ in \eqref{defSur}, is known under
several names.
In this note we provide a self-contained material gathering and extending results on this property scattered in the literature
and
illustrate possible applications.

\begin{definition} \label{defControl}
Consider a mapping  $F:X \tto Y$ between metric spaces $(X, d)$ and $(Y,\varrho)$ and a~point $(\bx, \by) \in X \times Y$.
\begin{itemize}
 \item[(B1)] $F$ is said to be {\em metrically semiregular at $(\bx,\by)$} when $\by \in F(\bx)$ and there is a constant $\kappa > 0$ along  with
a neighborhood $V$ of $\by$ in $Y$ such that
\begin{equation}
\label{defMHR}
      \dis \big(\bx, F^{-1}(y) \big) \leq \kappa \, \varrho (\by, y)
      \quad \mbox{for every} \quad   y \in V.
\end{equation}
 The infimum of $\kappa > 0$ for which there exists a neighborhood $V$ of $\by$ in $Y$ such that \eqref{defMHR} holds is called the {\em semiregularity modulus} of $F$
 at $(\bx,\by)$ and is denoted by $\semireg F (\bx,\by)$;
  \item[(B2)] $F$ is said to be  {\em open with a linear rate at $(\bar{x},\bar{y})$} when $\by \in F(\bx)$ and  there are positive constants $c$ and $\varepsilon$ such that
\begin{equation}\label{suro}
 \ball[\bar{y},ct] \subset F(\ball[\bar{x},t])
\quad  \mbox{for each} \quad    t  \in (0,\varepsilon).
\end{equation}
The  supremum of $c > 0$ for which there exists  a constant $\varepsilon>0$ such that (\ref{suro}) holds  is called the
{\em modulus of openness}
of $F$  at  $(\bar{x}, \bar{y})$ and is denoted by
$\lopen F(\bar{x},\bar{y})$.
\end{itemize}
\end{definition}

Properties (B1) and (B2) were studied by the third author in \cite{AK2009} (see also \cite{KruTha15}), where their equivalence was established (see Proposition~\ref{prop01} below) and the term \emph{semiregularity} was suggested for property (B1).
This property has been later used in \cite{AM2011, DS2012,Ude4} under the name \emph{hemiregularity}.
Following \cite{DmiMilOsm80}, property (B2) was referred to in \cite{AK2009} as \emph{$c$-covering}, while in the earlier paper \cite{Kru06} it was called simply \emph{regularity}.
This property can be found also in \cite{DS2012,book}.
In the recent survey by Ioffe \cite{IoffeSurvey}, the property is called \emph{controllability}, the concept stemming from the control theory.
The explicit definition of $\lopen F(\bar{x},\bar{y})$ can be found in \cite{Kru04,Kru05}, while its main components are present already in \cite{Kru00,Kru02}.
Note that thanks to the Robinson-Ursescu theorem, if $F$ has a closed convex graph, the openness (with a linear rate) at a point is equivalent to the openness around this point.

One can define the third (equivalent) property in terms of the inverse $F^{-1}$.
To the best of our knowledge, it first appeared in \cite[p.~34]{KKbook} under the name \emph{Lipschitz lower semicontinuity}. It was defined for $F^{-1}$ via inequality \eqref{defMHR}.
This property is called \emph{pseudo-calmness} in \cite{DS2012}, while  the term \emph{linear recession} is used in \cite{IoffeSurvey}.

 A (graphical) localization of a set-valued mapping $F: X \tto Y$ around the reference point $(\bx,\by) \in \gph F$ is any mapping $\widetilde F: X \tto Y$ such that $\gph \widetilde F = \gph F \cap (U \times V)$ for some neighborhood $U \times V$ of $(\bx, \by)$ in $X \times Y$.
 Using this notion we can define ``stronger" versions of the properties mentioned above.
\begin{definition} \label{defStrong}
Consider a mapping  $F:X \tto Y$ between metric spaces $(X, d)$ and $(Y,\varrho)$ and a point $(\bx, \by) \in X \times Y$. Then $F$ is said to be
\begin{itemize}
  \item[(S)]  {\em strongly metrically regular around $(\bx,\by)$} when $F$ is metrically regular at $(\bx,\by)$ and $F^{-1}$ has a localization around $(\by,\bx)$ which is nowhere multivalued;
  \item[(SA)]  {\em strongly metrically subregular at $(\bx,\by)$} when $F$ is metrically subregular at $(\bx,\by)$ and $F^{-1}$ has no localization around $(\by,\bx)$ that is multivalued at $\by$;
  \item[(SB)] {\em strongly metrically semiregular at $(\bx,\by)$} when $F$ is metrically semiregular at $(\bx,\by)$ and $F^{-1}$ has a localization around $(\by,\bx)$ which is nowhere multivalued.
\end{itemize}
\end{definition}

Clearly, (S)--(SA) are connected with (and can be defined by) the properties of the inverse $F^{-1}$. Indeed, (S) means that for each $\ell > \reg F(\bx,\by)$ there is a neighborhood $U \times V$ of $(\bx,\by)$ in $X \times Y$ such that the localization $V \ni y \longmapsto F^{-1}(y) \cap U$ is single-valued and Lipschitz continuous on $V$ with the constant $\ell$ \cite[Proposition 3G.1]{book}. While  (SA) means that for each $\ell > \subreg F(\bx,\by)$ there is a neighborhood $U\times V$ of $(\bx, \by)$ in $X \times Y$ such that
$$
d(\bx,x) \leq \ell \, \varrho(\by,y) \quad  \mbox{whenever} \quad
x \in U\quad \mbox{and} \quad y\in F(x)\cap V.
$$
Finally, $(SB)$ means that for each $\ell > \semireg F(\bx,\by)$ there is a neighborhood $U \times V$ of $(\bx,\by)$ in $X \times Y$ such that the localization $V \ni y \longmapsto F^{-1}(y) \cap U$ is single-valued and calm on $V$ with the constant $\ell$.
As in the
case of regularity, we omit the
word ``metrically" in the rest of the note, that is, we say that $F$ is subregular (semiregular, strongly regular, etc.) at/around $(\bx,\by)$.

 Note that the validity of both the weaker point-based properties does not imply the stronger one, that is, if $F$ satisfies $(A1)$ and $(B1)$ then $F$ does not need to be regular around the reference point (see Example~\ref{ex01}).

Now, we survey several well known results concerning regularity and (sub)regularity which are related to the ones presented in this note. Let us point out that in case of a single-valued mapping, denoted by  $f : X \to Y$, we do not mention the point $\by = f(\bx)$ in all the above definitions, that is, we write $\sur f(\bx)$, $\reg f(\bx)$, etc., instead of $\sur f(\bx,f(\bx))$, $\reg f(\bx,f(\bx))$, etc.; and if the corresponding modulus is independent of $\bx$ then we omit $\bx$ as well.

Suppose that $X$ and $Y$ are Banach spaces and  $A: X \to Y$  is a continuous linear operator. Then the Banach-Schauder open mapping theorem and the linearity of $A$  imply (cf. \cite[Theorem~1.104 and Proposition 1.106]{JPP}, \cite[Proposition 5.2]{AM2011}) that: {\it $A$ is regular around any point $\Leftrightarrow$ $A$ is semiregular at any point  $\Leftrightarrow$ $A$ is surjective; moreover
$$
 \semireg A = \reg A \quad \mbox{and} \quad  \sur A = \sup\{\varrho > 0:  A(\ball_X) \supset \varrho \ball_Y\} = \inf\{ \|A^*y^*\|: y^* \in \mathbb{S}_{Y^*}\},
$$}
where $A^*$ is the adjoint (dual) operator to $A$ acting between the dual spaces $Y^*$ and $X^*$ of $Y$ and $X$, and $\ball_Z$ and $\mathbb{S}_Z$ denote the closed unit ball and the unit sphere  in a normed space $Z$, respectively.
This is a particular case of Proposition~\ref{prop02} (iv). If $A$ is invertible, then $ \sur A = 1/\|A^{-1}\|$. For  a~real $m$-by-$n$ matrix  $A \in \R^{m \times n}$,
$ \sur A$
equals to the least singular value of $A$. Using the Banach-Schauder theorem again, if $A$ has
closed range, then it is subregular at any point; and if, in addition, $A$ is injective then it is strongly subregular everywhere. Note that {\it both} the statements fail without the
closedness
assumption (see \cite[Example 2.7]{CDK}). In general, $A$ is strongly subregular everywhere if and only if
$\kappa:=\inf_{h \in \mathbb{S}_X} \|Ah\| > 0$; moreover $\subreg A = 1/\kappa$.
If the dimension of $X$ is finite, then $\kappa > 0$ if and only if $A^{-1}(0) = \{0\}$, that is, $A$ is injective.

Using the above notation, for a non-linear mapping we have the following result:
\begin{theorem}\label{Gr}  Consider a mapping $f:X \to Y$ defined around  a point $\bx \in X$ and a continuous linear  mapping $A:X \to Y$.
\begin{itemize}
 \item[(i)]
Then
$ \sur f (\bx) \geq \sur A - \lip (f-A)(\bx)$.
If, in addition, the mapping
$A$ is invertible and $  \lip (f-A)(\bx) <  \sur A$, then $f$ is strongly regular at $\bx$ and  $\sur f (\bx) \geq 1/\|A^{-1}\| - \lip (f-A)(\bx)
$ ( $> 0$).
 \item[(ii)] If $A$ is strongly subregular (everywhere) and  $\calm (f-A)(\bx) <  \psopen A$, then $f$ is strongly subregular at $\bx$ and
 $  \psopen f (\bx) \geq \psopen A - \calm (f-A)(\bx)$ ( $> 0$).
\end{itemize}

\end{theorem}
 Theorem~\ref{Gr} is a particular case of the well known fact that (strong) regularity as well as strong subregularity are stable with respect to a single-valued perturbation (see Theorem~\ref{thmPerturb} below). Part (i) was proved by Graves \cite{G} and Graves-Hildebrand \cite{HG}.
 More precisely, Graves proved that $\lopen f (\bx) \geq \sur A - \lip (f-A)(\bx) > 0$, which is weaker. As observed in \cite{g_ad} a~slight modification of the original proof yields the (stronger) version above.  If $A$ is the strict derivative\footnote{Sometimes called strong derivative \cite{Pourciau}.}
of $f$ at $\bx$, that is, when  $\lip (f-A)(\bx) = 0$, then we have $\sur f (\bx) = \sur A$. This is the case, for example, if $f$ is (Gateaux) differentiable in a vicinity of $\bx$ and the derivative mapping  $x \longmapsto Df(x)$ is continuous at $\bx$ as a mapping from $X$ into $\mathcal{L}(X,Y)$, the space of all  linear bounded operators  from $X$ into $Y$. In fact, the weak Gateaux differentiability is enough.
In particular, the Lyusternik theorem \cite{L}, proved before the Graves theorem,  follows from Theorem~\ref{Gr}.
On the other hand, assume that $X := \R^n$ and $Y:=\R^m$. If $f$ is strictly differentiable at $\bx$, then there is a neighborhood $U$ of $\bx$ such that $f$ is Lipschitz continuous on $U$. Let $D \subset U$ be the set of all $x \in U$ such that $f$ is Fr\'echet differentiable at $x$. Then $D$ has full Lebesgue
measure by the Rademacher theorem. Moreover, the Jacobian mapping $D \ni x \longmapsto \nabla f(x) \in \R^{m \times n}$ is continuous at $\bx$ \cite[Lemma 5.1]{Pourciau}. However, this does not imply that $f$ is differentiable on any neighborhood of $\bx$ (\cite[p. 324]{Pourciau} or \cite[p.35]{book}). If $f$ is differentiable in a vicinity of $\bx$ then $f$ is strictly differentiable at $\bx$ if and only if $\nabla f$ is continuous at $\bx$ \cite[Proposition 1D.7]{book}.
Theorem~\ref{Gr} (ii) which can be found as \cite[Theorem 2.1]{CDK}, for example,  fails when (non-strong) subregularity is considered \cite[p. 201]{book}.


To check the regularity of the mapping in question we have the following regularity criterion  \cite[Corollary 1]{fp}, \cite[Theorem 1b]{i2000}, \cite[Proposition 2.1]{CF}.

\begin{proposition}\label{criterion1}
Let $(X,d)$ be a complete metric space and $(Y,\varrho)$ be a metric space,
let $\bar x\in X$ be given, and let $g: X\to Y$ be a continuous mapping, whose domain is all of $X$.
Then $\sur g(\bar x)$ equals to the supremum of all $c > 0$ for which
there is $r > 0$ such that for all $(x,y) \in \ball(\xx,r) \times \big(\ball(g(\xx),r)\setminus \{g(x) \} \big)$ there is a point $x'\in X$ satisfying
\begin{eqnarray}\label{hvezda}
 c\,d(x',x) < \varrho(g(x),y)-\varrho(g(x'),y).
\end{eqnarray}
\end{proposition}

More precisely, Fabian and Preiss \cite{fp} proved only a sufficient condition guaranteeing that $\lopen g(\bx) > 0$. The full version (for set-valued mappings) was shown independently by Ioffe \cite{i2000}.
As in the
case of Theorem~\ref{Gr}, only a tiny modification of the original proof from \cite{fp} yields the statement above (see \cite{CF}).
Although Proposition~\ref{criterion1} is formulated for a single-valued function, it is well-known that the study of regularity  properties for a set-valued mapping $F: X \rightrightarrows Y$ can always be reduced to
the study of the corresponding property for a simple
single-valued mapping, namely, the restriction of the canonical projection from $X \times Y$ onto $Y$,
that is, the assignment ${\rm gph}\, F\ni(x,y)\longmapsto y\in Y$ (e.g., see \cite[Proposition 3]{i2000}). Using this, one gets the following statement for set-valued mappings.

\begin{theorem} \label{crit2}
Let $(X,d)$
 and
$(Y,\varrho)$ be metric spaces and let $F: X\rightrightarrows Y$ be a set-valued mapping having  a localization around $(\bar{x}, \bar{y}) \in \mbox{\rm gph} \, F$ with a complete graph.
Then $\sur F(\bar{x},\bar{y})$ equals to the supremum  of all $c > 0$ for which
there are $r > 0$ and $\alpha \in (0, 1/c) $ such that for any $(x,v) \in \mbox{\rm gph} \, F \cap \big(\ball(\bar{x}, r) \times \ball(\bar{y}, r) \big)$ and any $y \in \ball(\bar{y}, r) \setminus  \{v\}$ there is a pair $(x',v') \in \mbox{\rm gph} \, F$ satisfying
    \begin{eqnarray}\label{hvezda2}
    c \max\{d(x,x'), \alpha \varrho(v,v') \} < \varrho(v,y)- \varrho(v',y).
    \end{eqnarray}
\end{theorem}

It is elementary to check that $\psopen F(\bar{x},\bar{y})$ equals to
the \emph{subregularity constant} of $F$ at $(\bx,\by)$ defined in \cite{Kru15} as
\begin{equation}\label{AK15}
  \liminf_{x\to\bx, \ x\notin F^{-1}(\by)} \frac{\dist(\by,F(x))}{\dist(x,F^{-1}(\by))},
\end{equation}
with the convention that the limit in \eqref{AK15} is $\infty$ when $\bx$ is an internal point in $F^{-1}(\by)$.
When $\bx$ is an isolated point in $F^{-1}(\by)$, then $\psopen F(\bar{x},\bar{y})$ coincides with the
\emph{steepest displacement rate}  at $(\bx,\by)$ defined by
Uderzo
in \cite{u} as
\begin{equation}\label{P5.1.1-1}
|F|^{\downarrow}(\bx,\by) := \liminf_{x\to\bx} \frac{\dist(\by,F(x))}{d(\bx,x)},
\end{equation}
with the convention that the limit in \eqref{P5.1.1-1} is $\infty$ when $\bx$ is an isolated point in the domain of $F$.
The inequality $|F|^{\downarrow}(\bx , \by)>0$ is equivalent to the \emph{strong} subregularity of $F$ at $(\bx,\by)$.

There is a similar statement to Theorem~\ref{crit2} guaranteeing
the (strong)
subregularity.
The next theorem combines a portion of  \cite[Corollary~5.8]{Kru15} (with condition (d)) and \cite[Theorem~5.3]{CDK}.
The latter one was formulated in \cite{CDK} for Banach spaces, but its proof remains valid in the present setting.

\begin{theorem} \label{thmCrit}
Let $(X,d)$ and
$(Y,\varrho)$ be metric spaces and let $F:X \tto Y$ be a set-valued mapping having a localization around $(\bar{x}, \bar{y}) \in \mbox{\rm gph} \, F$ with a complete graph.  Then
$\psopen F(\bar{x},\bar{y})$ (respectively,
$|F|^{\downarrow}(\bx,\by)$)
equals to the supremum
of  $c > 0$ for which there exists  $r>0$ such that for any $(x,y)\in\gph F$ with
$x\notin F^{-1}(\by)$ and $d(x,\bx)<r$ (respectively, $0< d(x,\bx)<r$)
and $\varrho(y,\by) <r$, there is  a pair $(u,v)\in\gph F\setminus\{(x,y)\}$ satisfying
\begin{equation}\label{eqCrit}
 c \max\{ d(u,x), r\varrho(v,y)\} < \varrho(y,\by) - \varrho(v,\by).
\end{equation}
\end{theorem}


 Note that (sufficient) conditions for (non-strong) subregularity and  semiregularity are much more involved because of their instability  with respect to calm (or Lipschitz) single-valued perturbations (see counterexamples \cite[pp. 200--201]{book}).
More precisely, for these two properties,
the analogues of the following statement (see \cite[Theorems~5E.1 and 5F.1]{book} and \cite[Corollary 2.2]{CDK}) fail
without additional assumptions.

\begin{theorem} \label{thmPerturb}
Let $(X,d)$ be a complete metric space, $(Y,\varrho)$ be a linear metric space with a shift-invariant metric, and $(\bx,\by) \in X \times Y$. Consider a mapping $g :X \to Y$ defined around $\bx$ and a~mapping $F:X \tto Y$ such that $ \by \in F(\bx)$.
\begin{itemize}
 \item[(i)] If $F$ is (strongly) regular around $(\bx,\by)$ and
 $\lip g (\bx) <  \sur F(\bx,\by)$, then
so is $g+F$ around $(\bx,g(\bx) + \by)$ and
$$
 \sur(g+F)(\bx,g(\bx) + \by) \geq  \sur F(\bx,\by) - \lip g (\bx)   > 0.
$$
\item[(ii)] If
$F$ is strongly subregular at $(\bx,\by)$ and
$ \calm g (\bx) <  \psopen F(\bx, \by)$,
then so is $g+F$ at  $(\bx,g(\bx) +\by)$ and
$$
 \psopen(g+F)(\bx, g(\bx) + \by) \geq  \psopen F(\bx, \by) - \calm g (\bx)   > 0.
$$
\end{itemize}
\end{theorem}
\noindent The above statement fails if a perturbation is set-valued (see \cite[Example 5I.1]{book} and \cite[p. 5]{CDK}).

Regularity as well as strong (sub)regularity are known to play a key role  in  the local convergence analysis for Newton-type iterative schemes for solving a generalized equation, introduced by Robinson in \cite{rob}, which reads as:
\begin{equation} \label{GE}
\mbox{Find}\quad x\in X \quad \mbox{such that}\quad f(x)+F(x) \ni 0,
\end{equation}
where $X$ and $Y$ are (real) Banach spaces, $f: X \to Y$ is a single-valued (possibly nonsmooth) mapping, and $F: X \tto Y$ is a set-valued mapping with closed graph.  This model has been used to describe
in a unified way various problems such as equations (when $F \equiv 0$), inequalities (when $Y = \R^n$ and
$F \equiv \R^n_+$), variational inequalities (when $Y=X^*$ and $F$ is the normal cone mapping corresponding to a closed convex subset of  $X$ or
more broadly the subdifferential mapping of a~convex function on $X$).

The Newton iteration  for \eqref{GE} with a smooth function $f$,   also known as the Josephy-Newton method \cite{josephy},  has the form
\begin{equation}\label{NJ}
f(x_k) + f'(x_k)(x_{k+1}-x_k) +F(x_{k+1})\ni 0 \quad \mbox{for each} \quad  k\in \N_0:=\{0\}\cup \N \mbox{ and a given } x_0 \in X.
\end{equation}
 From the numerical point of view, it is clear that the auxiliary inclusions above cannot be solved exactly because of the finite precision arithmetic and rounding errors.  Moreover, it can be much quicker to find an inexact solution at each step which has a sufficiently small residual. Various (in)exact methods were proposed in the literature (see \cite{ISbook} for an in-depth study and a vast bibliography, or \cite{KKbook} and references therein).  In order to represent inexactness, Dontchev and Rockafellar proposed
 in \cite{in}  an inexact version of  the iteration
 \eqref{NJ} in which,  for
given $k \in \N_0$ and $x_k \in X$, the next iterate $x_{k+1} \in X $ is determined as a  {\em coincidence point } of the mapping on the left-hand side of \eqref{NJ} and a mapping {$R_k: X \times X \tto Y$} which models inexactness, that is,
\begin{equation} \label{inNJ}
{\big(}f(x_k) + f'(x_k)(x_{k+1}-x_k) +F(x_{k+1}){\big)}\cap R_k(x_k, x_{k+1})\neq \emptyset.
\end{equation}

Now, we describe the structure of our note in detail as well as the relation of the results presented and the existing ones.
In Section~\ref{secCompare}, we recall that there is a clear link between semiregularity and openness at a point similar to \eqref{eqBasic} and \eqref{eqBasicA}. Next, we remark that, for particular mappings, semiregularity can imply regularity and that the corresponding moduli can be easily computed (as in the case of a continuous linear operator). On the other hand, we provide examples illustrating the differences. Proposition~\ref{prop02}(ii) slightly generalizes known results that the usual openness implies the \emph{linear} openness under a certain ``convexity" assumption  on the graph of the mapping under consideration.

In Section~\ref{secCriteria}, we
discuss both primal and dual  \emph{infinitesimal}  conditions. More precisely, new slope-based necessary as well as sufficent conditions are obtained (Theorem~\ref{AK-T1.9}) and the dual necessary condition is recalled (Theorem~\ref{thmDual}). Theorem~\ref{thmKaluza}, which seems to be new, is a finite-dimensional analogue of Theorem~\ref{Gr} and its corollaries generalize existing results in one, in our opinion, very important direction for applications - the usual openness is strengthened to \emph{linear} openness. We show that a similar approach yields a statement for set-valued mappings satisfying certain ``strong monotonicity/ellipticity" assumptions (Theorem~\ref{thmKaluzeSetValued}) and present corollaries correcting some statements from the literature (cf. Remark~\ref{rmSVKAL}).

In Section~\ref{secSum}, we prove general  necessary as well as sufficient conditions in the spirit of Theorems~\ref{crit2} and \ref{thmCrit}, which are known to provide short and elegant proofs of various regularity statements in the literature \cite{IoffeSurvey, CF}. To prove set-valued versions of these conditions, we use the standard ``projection trick" described above. We also provide a completely new,  elementary, and short proof that the sum of two set-valued mappings is semiregular provided that one mapping is regular while the other is pseudo-Lipschitz. Note that neither  error bounds nor the slopes are needed; it is enough to use a statement for single-valued mappings, and, more importantly, the key steps and the length of the proof remain the same when one applies its set-valued version.

In Section~\ref{secNewton},
we analyze an  inexact Newton-type iteration for the case when the function $f$ in \eqref{GE} is not necessarily differentiable.
Specifically, we introduce a  mapping    ${\cal H}: X\tto
{\cal L}(X,Y)$  viewed as a {\em generalized set-valued derivative} of the function $f$, and consider the following iteration:
{\em Given an index
$k \in \N_0$ and a point $x_k \in X$,  choose any  $A_k \in {\cal H}(x_k)$ and then
 find $x_{k+1} \in X$  satisfying }
\begin{equation}
\label{niN}
 {\big(} f(x_k)+ A_k(x_{k+1}-x_k)+F(x_{k+1}){\big)}\cap R_k(x_k, x_{k+1}) \neq \emptyset.
 \end{equation}
Semiregularity of $f(x_k)+ A_k(\cdot -x_k)+F - R_k(x_k, \cdot)$ turns out to play a fundamental role in the existence of the next iterate close enough to the current one. The case when the mappings $R_k$  depend on the current iterate $x_k$ only, was studied in \cite{CDG}. The proof of the convergence result is divided into two steps.  Step 1 uses the classical statements and establishes  \emph{uniformity} of the assumed regularity, that is, that the constants and neighborhoods can be taken the same. The semiregularity of the sum is needed in Step 2 which, of course, can be done via a (complicated) double fixed-point theorem as in \cite{in}. We show that the perturbation result is strong enough to obtain the conclusion (as in the usual case) obtaining in this way a completely different proof from \cite{in}.

\paragraph{
Notation and terminology.}
 When we write ${f: X\to Y}$ we mean that $f$ is a (single-valued) mapping acting from $X$ into $Y$
while  ${F:X \tto Y}$ is a  mapping from $X$ into $Y$ which may be set-valued. The set $\dom F:=\{ x \in X:\; F(x)\neq\emptyset\}$
is the {\it domain} of $F$, the  \emph{graph} of $F$ is the set $\gph F: = \{
(x,y)\in X\times Y: \ y\in F(x)\}$ and the \emph{inverse} of $F$ is the mapping $Y \ni y \longmapsto \{x \in X : \ y\in F(x)\}=: F^{-1}(y) \subset X$;
thus $F^{-1} :Y\rightrightarrows X$.
In any metric space, $\ball[x,r]$ denotes the closed ball centered at $x$ with a radius $r > 0$ and $\ball(x,r)$ is the corresponding open ball.
$\ball_X$ and $\mathbb{S}_X$ are respectively the \emph{closed unit ball} and the \emph{unit sphere} in a
normed space $X$. The \emph{distance from a point} $x$ to a subset $C$ of a metric space $(X,d)$ is $\dis(x,C):= \inf \{d(x,y): \ y \in C \}$. We use the convention that $ \inf \emptyset: = \infty$ and as we work with non-negative quantities we set $\sup \emptyset: = 0$. If a set is
a
singleton we identify it with its only element, that is, we write $a$ instead of $\{ a \}$.
The symbol ${\mathcal L}(X,Y)$ denotes the space of all linear bounded operators from a Banach space $X$ into a~Banach space $Y$. Then $\R^{m\times n} := \mathcal{L}(\R^n,\R^m)$ and $X^* := {\mathcal L}(X,\mathbb{R})$.
Given $A \in {\mathcal L}(X,Y)$, the operator $A^*:Y^* \to X^*$ denotes the adjoint (dual, transpose) operator to $A$. The transpose of  a matrix $A \in \R^{m\times n}$ is $A^T \in \R^{n \times m}$. Given a set  $\mathcal{A}$ in ${\mathcal L}(X,Y)$,  the {\it measure of noncompactness} $\chi (\mathcal{A})$  of $\mathcal{A}$ is defined as
$$
\chi (\mathcal{A}) := \inf\big\{  r > 0: \ \mathcal{A} \subset \mathcal{F} + r\ball_{\mathcal{L}(X,Y)} \mbox{ for some finite }  {\mathcal F} \subset  \mathcal{A} \big\}.
$$
Given an extended real-valued function $\varphi: X \to \R \cup \{\infty\} $ and a point $x \in X$, the  \emph{limes inferior} of $\varphi$ at $x$ is defined by
$$
 \liminf_{u\to x} \varphi (u) :=\sup_{r>0} \inf_{u\in \ball(x,r)} \varphi(u).
$$

\section{Relationship among regularity concepts} \label{secCompare}

Let us start with a simple observation \cite[Proposition 2.4]{DS2012} and \cite[Theorem 6(i)]{AK2009}:

\begin{proposition} \label{prop01}
Consider a mapping  $F:X \tto Y$ between metric spaces $(X,d)$ and $(Y, \varrho)$ and a~point $(\bx, \by) \in \gph F$. Then 
\begin{equation} \label{eqBasicB}
 \lopen F(\bar{x},\bar{y}) \cdot \semireg F (\bx,\by) = 1. 
\end{equation}
\end{proposition}

The relationship among various properties is summarized in the following statement:

\begin{proposition} \label{prop02}
Consider a mapping  $F:X \tto Y$ between metric spaces $(X,d)$ and $(Y, \varrho)$ and a~point $(\bx, \by) \in X \times Y$. Then
\begin{itemize}
\item[(i)] $ \lopen F(\bar{x},\bar{y}) \geq \liminf\limits_{(x,y) \to (\bx,\by), \ y\in F(x)} \lopen F({x}, {y}) \geq   \sur F (\bx,\by)$.
\item[(ii)] Suppose that $X$ and $Y$ are normed spaces and that $F$ has a locally star-shaped graph at  $(\bar x, \bar y)$, that is, there is $a \in (0,1]$ such that   $  (1-t) \, (\bar x, \bar y) + t \gph F \subset \gph F$ for each $ t \in [0,a]$.
 If there are positive constants $\alpha$ and $\beta$ such that
\begin{equation} \label{eqalphabeta}
  \ball [ \bar y, \beta] \subset F(\ball[\bar{x},\alpha]),
\end{equation}
  then  $\lopen F(\bar{x},\bar{y}) \geq \beta/\alpha$.
\item[(iii)] If $X$ and $Y$ are normed spaces and $F$ has a convex graph then  $ \lopen F(\bar{x},\bar{y}) = \sur F (\bx,\by)$.
\item[(iv)] If $X$ and $Y$ are Banach spaces and $F$ is a closed convex process, that is, $\gph F$ is a closed convex cone in $X \times Y$, then
$$
 \lopen F(0,0) = \sur F (0,0) = \sup\{\varrho > 0: F(\ball_X) \supset \varrho \ball_Y \} = \inf\{\|x^*\|: x^* \in F^*(\mathbb{S}_{Y^*}) \},
$$
where $F^*:Y^* \to X^*$ is the adjoint process to $F$ defined by
$$
  F^*(y^*) = \{ x^* \in X^*: \langle x^*, x \rangle \leq
\langle y^*, y \rangle \mbox{ for each } (x,y) \in \gph F\}
.$$
\end{itemize}
\end{proposition}

\begin{proof} Statement (i) follows immediately from the definitions of $\sur F(\bx, \by)$  and the limes inferior, while (iv) is \cite[Theorem 7.9]{IoffeSurvey}.
Assume without any loss of generality that $\bx = 0$ and $\by =0$.

(ii)  By assumption, there is $a  \in (0,1]$ such that
$  \tau  \gph F  \subset \gph F  $ for each  $ \tau \in [0,a]$. Then \eqref{eqalphabeta} implies that
$$
  \tau \beta \ball_Y \subset F(\tau\alpha \ball_X) \quad \mbox{for each} \quad \tau \in [0,a].
$$
 Indeed, fix any such $\tau$. Pick an arbitrary $y \in \tau\beta \ball_Y$. Then $v:=y/\tau \in \beta \ball_Y$. By \eqref{eqalphabeta}, there is $u \in X$ such that $v \in F(u)$ and $\|u\| \leq \alpha$. Then $x:=\tau u \in \tau\alpha \ball_X$. Moreover, $(x,y) = \tau (u,v) \in \tau \gph F \subset \gph F$. Thus $y \in F(x)$.

 \noindent Set $c:= \beta/\alpha$ and $\varepsilon := \alpha a$. Fix any $t \in (0, \varepsilon)$. Then $\tau:= t/\alpha \in (0, a)$, and consequently,
$$
  F(t \ball_X) = F(\tau \alpha \ball_X) \supset \tau \beta \ball_Y = ct \ball_Y.
 $$
 (iii)  By (i), it suffices to show that $\lopen F(0,0) \leq \sur F(0,0)$. Fix arbitrary $c$, $\tilde c \in (0,\lopen F(0,0))$ with $c < \tilde c$.  Find $\alpha \in (0,1)$ such that $\tilde c \alpha \ball_Y \subset F(\alpha \ball_X)$, and then $r > 0$ such that $c(\alpha + r) +r < \tilde c \alpha$.
 Fix any $(x,y) \in \gph F$ with $\|x\| \leq r$ and $\|y\| \leq r$. Then
 \begin{eqnarray*}
  \ball[y,c(\alpha+r)] \subset (c(\alpha+r) + r) \ball_Y \subset \tilde c \alpha \ball_Y  \subset F(\alpha  \ball_X) \subset F(\ball[x,\alpha + r]).
 \end{eqnarray*}
 As in the proof of (ii), with $a:=1$, $\beta :=c (\alpha+r)$, and $(\bx, \by,\alpha)$ replaced by $(x,y,\alpha+r)$,  we conclude that for  any $t \in (0,\alpha+r)$ we have $ \ball[y,ct] \subset F(\ball[x,t])$. Since $\alpha$ and $r$ are independent of $(x,y)$, we obtain that   $ \sur F(0,0) \geq c$. Letting $c \uparrow \lopen F(0,0)$ we get the desired estimate.

\end{proof}

To illustrate the difference between the regularity properties we provide the following examples.

\begin{example} \label{ex01} \rm
 Consider a function $f: \R \to \R$ defined by
 $$
  f(x) =
  \left\{
     \begin{array}{ll}
       x + \frac{x^3}{|x|} \left|\sin\left(\frac{1}{x} \right)\right| &  \mbox{ if} \quad x \neq 0,\\
       0                   &  \mbox{ if} \quad x = 0.\\
     \end{array}
  \right.
 $$
 Then $f$ is locally Lipschitz around $0$, Fr\'echet differentiable at $0$ (and almost everywhere) but not strictly differentiable at $0$, and there is no neighborhood $U$ of $0$ such that $f$ is differentiable on $U$. Moreover, $f$ is semiregular (not strongly), strongly subregular at $0$, and   $\sur f(0) = \liminf\limits_{x \to 0} \lopen f(x) = 0$, while $f'(0)= \lopen f(0) = \psopen f(0) = 1$. In particular, the first inequality in Proposition~\ref{prop02} (i) is strict.
\end{example}

\begin{example} \rm
Consider a function $f:\R\tto\R$ given by
$$
f(x):=
\begin{cases}
x,&\mbox{if } x\le0,
\\
x-\frac{1}{n},&\mbox{if } \frac{1}{n}<x\le\frac{1}{n-1}, \quad n=3, 4, \ldots,
\\
x-\frac{1}{2},&\mbox{if } x>\frac{1}{2},
\end{cases}
$$
and its epigraphical mapping $F(x):=\{y\in\R : y\ge f(x)\}$, $x\in\R$.
It is easy to check that $\lopen F(x,y)=\infty$ if $y>f(x)$ and $\lopen F(x,y)=1$ if $y=f(x)$.
Hence,
$$
\lim\limits_{r\downarrow 0} \inf{ \{ \lopen F(x,y): \ (x,y) \in \gph F \cap \big(\ball(0,r) \times \ball(0,r) \big) \}}=1.
$$
Take any $r>0$ and $\varepsilon>0$, and choose an index $n\in\N$ such that $x_n:=\frac{1}{n}+\frac{1}{n^2}<r$ and $t_n:=\frac{1}{n}<\varepsilon$.
Then $y_n:=f(x_n)=\frac{1}{n^2}<r$ and
$$\sup \left\{c>0: \ball[y_n,ct_n] \subset F(\ball[x_n,t_n]) \right\}=\frac{1}{n}.$$
Hence,
$$
\inf_{(x,y) \in \gph F \cap \big(\ball(0,r) \times \ball(0,r) \big)}\inf_{t\in(0,\varepsilon)} \sup\{c>0: \ball[y,ct] \subset F(\ball[x,t])\}=0,
$$
and therefore $\sur F(0,0) = 0$.  Consequently, the second inequality in Proposition~\ref{prop02} (i) is strict.
\end{example}

\section{Primal and dual infinitesimal conditions} \label{secCriteria}

\def\SeR{{\rm SeR}[F](\bx,\by)}
\def\SSeR{\overline{|\nabla{F}|}{}_{\rm SeR}^\diamond(\bx,\by)}
\def\lsc{lower semicontinuous}
\def\usc{upper semicontinuous}
\def\LHS{left-hand side}
\def\RHS{right-hand side}
\newcommand{\de}{\delta}
\newcommand{\la}{\lambda}
\newcommand{\ga}{\gamma}

It is easy to check that $\lopen F(\bar{x},\bar{y})$ equals to
\begin{equation}\label{AK18}
\liminf_{y\to\by, \, y\notin F(\bx)} \frac{\varrho(y,\by)}{\dist(\bx,F^{-1}(y))},
\end{equation}
with the convention that the limit in \eqref{AK18} is $\infty$ when $\by$ is an internal point in $F(\bx)$.

\begin{theorem} \label{AK-T1.9}
Let $(X,d)$ be a metric space, $(Y,\varrho)$ a complete metric space, and let $F:X \tto Y$ be a set-valued mapping such that $(\bar{x}, \bar{y}) \in \mbox{\rm gph} \, F$ and the function $y\longmapsto\dist(\bx,F^{-1}(y))$ is upper semicontinuous near $\by$.
Set
\begin{equation}\label{AK-T1.9-0}
\varphi(y):=
\begin{cases}
\displaystyle
\frac{\varrho(y,\by)}{\dist(\bx,F^{-1}(y))},& \mbox{if } y\ne\by,
\\
0,&\mbox{otherwise},
\end{cases}
\end{equation}
\begin{equation}\label{AK-T1.9-1}
\SSeR:=\liminf_{y\to\by, \ y\notin F(\bx)} \varrho(y,\by)\sup_{v\ne y} \frac{\displaystyle\varphi(y)- \varphi(v)}{\varrho(y,v)}.
\end{equation}
Then
\begin{equation}\label{AK-T1.9-2}
\frac{1}{2}\SSeR\le  \lopen F(\bar{x},\bar{y}) \le\SSeR.
\end{equation}
In particular, if numbers $c > 0$ and  $r>0$ are
such that, for any $y\in\ball[\by,r]\setminus F(\bx)$, there is  a~vector $v\in Y$ satisfying
\begin{equation*}
\varrho(y,\by)\left(\varphi(y)- \varphi(v)\right)>c\,\varrho(y,v),
\end{equation*}
then $ \lopen F(\bar{x},\bar{y}) \ge c/2$.
\end{theorem}
\begin{proof}
Clearly, \begin{equation}\label{AK-T1.9P1}
 \lopen F(\bar{x},\bar{y}) =\liminf_{y\to\by, \ y\notin F(\bx)} \varphi(y).
\end{equation}
We prove the first inequality in \eqref{AK-T1.9-2}.
If $ \lopen F(\bar{x},\bar{y})=\infty$, the inequality holds trivially.
Let $ \lopen F(\bar{x},\bar{y})<\ga<\infty$.
We are going to show that $\SSeR\le2\ga$.
Note that $\varphi$ is \lsc\ near $\by$ and $\varphi(y)\ge0$ for all $y\in Y$.
Choose a number $\de>0$ such that $\varphi$ is \lsc\ on $\ball[\by,3\de]$.
By \eqref{AK-T1.9P1}, there exists a point $y'\in\ball[\by,\de]$ such that $y'\notin F(\bx)$ and $\varphi(y')<\ga$.
Set $\de':=\varrho(y',\by)$.
Then $0<\de'\le\de$.
Employing the Ekeland variational principle, we find a point $\hat y\in\ball(y',\de')$ such that $\varphi(\hat y)\le\varphi(y')$ and
\begin{equation}\label{AK-T1.9P2}
\varphi(\hat y)\le\varphi(v)+\frac{\ga}{\de'}\varrho(\hat y,v)
\end{equation}
for all $v\in\ball[\by,3\de]$.
Since $\varphi(\hat y)\le\varphi(y')<\infty$, in view of \eqref{AK-T1.9-0}, we have either $\hat y\notin F(\bx)$ or $\hat y=\by$.
At the same time,
\begin{equation*}
\varrho(\hat y,\by)\ge \varrho(y',\by)-\varrho(\hat y,y')>0.
\end{equation*}
Thus, $\hat y\ne\by$, and consequently, $\hat y\notin F(\bx)$.
Note that
\begin{equation*}
\varrho(\hat y,\by)\le \varrho(\hat y,y')+\varrho(y',\by) <2\de'.
\end{equation*}
If $v\notin\ball[\by,3\de]$, then
\begin{align*}
\varphi(\hat y)\le\varphi(y')<\ga \le\frac{\ga}{\de'}(3\de-2\de') <\frac{\ga}{\de'}(\varrho(v,\by)-\varrho(\hat y,\by))  \leq \frac{\ga}{\de'}\varrho(\hat y,v) \le\varphi(v)+\frac{\ga}{\de'}\varrho(\hat y,v).
\end{align*}
Hence, inequality \eqref{AK-T1.9P2} holds true for all $v\in Y$, and consequently,
\begin{equation*}
\varrho(\hat y,\by)\sup_{v\ne\hat y}\frac{\varphi(\hat y) -\varphi(v)}{\varrho(\hat y,v)} <2\de'\frac{\ga}{\de'}=2\ga.
\end{equation*}
Thus,
\begin{equation*}
\inf_{y\in\ball(\by,2\de)\setminus F(\bx)} \varrho(y,\by)\sup_{v\ne y} \frac{\varphi(y)-\varphi(v)}{\varrho(y,v)}<2\ga.
\end{equation*}
Passing to the limit as $\de\downarrow0$, we obtain $\SSeR\le2\ga$.
Since $\ga>  \lopen F(\bar{x},\bar{y})$ is arbitrary, the first inequality in \eqref{AK-T1.9-2} is proved.
Given any $y\ne\by$, we have
\begin{equation*}
\varrho(y,\by)\sup_{v\ne y} \frac{\varphi(y)-\varphi(v)}{\varrho(y,v)} \ge \varrho(y,\by) \frac{\varphi(y)-\varphi(\by)}{\varrho(y,\by)}=\varphi(y).
\end{equation*}
In view of the representations \eqref{AK-T1.9-1} and \eqref{AK-T1.9P1}, this proves the second inequality in \eqref{AK-T1.9-2}.
\end{proof}
\begin{remark} \rm
The second inequality in \eqref{AK-T1.9-2} is valid without the assumptions of the completeness of $Y$ and upper semicontinuity of the function $y \longmapsto\dist(\bx,F^{-1}(y))$. The last property holds, for example,  if $F^{-1}$ is \lsc, that is, when  $F$ is open at the corresponding reference point.
 \end{remark}

Let $X$ and $Y$ be normed spaces. Given a set $\Omega \subset X$ and a point $\bx \in \Omega$,  the \emph{Fr\'echet normal cone} to $\Omega$ at  $\bx$, denoted by  $\widehat{N}_{\Omega}(\bx)$,
is the set of all $x^* \in X^*$ such that for every $\varepsilon > 0$ there exists $\delta > 0$ such that
$$
  \langle x^*, x - \bx \rangle \leq \varepsilon \|x - \bx \| \quad \mbox{whenever} \quad x \in \Omega \cap \ball(\bx,\delta).
$$
For a mapping  $F : X \tto Y$  with $(\bx,\by) \in \gph F$,
the  \emph{Fr\'echet coderivative}  of $F$ at $(\bx,\by)$ acts from $  Y^* $ to the subsets of $X^*$ and is defined as
$$
Y^* \ni y^* \longmapsto  \widehat{D}^*F(\bx,\by)(y^*):= \left\{x^* \in  X^*: \ (x^*,-y^*) \in  \widehat{N}_{\gph F} (\bx,\by) \right\}.
$$

We have the following dual necessary condition for semiregularity
\cite[Theorem 6 (iv)]{AK2009}.
\begin{theorem} \label{thmDual}
Consider a mapping  $F:X \tto Y$ between normed spaces $X$ and $Y$ and a point $(\bx, \by) \in
\gph F$. Then
\begin{equation*}
  \lopen F(\bar{x},\bar{y}) \le\inf_{y^*\in\mathbb{S}_{Y^*}}\{\|x^*\|:\ x^*\in\widehat{D}^*F(\bx,\by)(y^*)\}.
\end{equation*}
Hence,
if $F$ is semiregular at $(\bx,\by)$ then
$$
  \widehat{D}^*F^{-1}(\by , \bx) (0) = \{0\}.
$$
\end{theorem}

In finite dimensions, using Brouwer's fixed point theorem, we get:

\begin{theorem} \label{thmKaluza}
Consider a point $\bar x \in \R^n$ along with a mapping $f : \R^n \to \R^m$ which is both defined and continuous in a vicinity of $\bar x$. Suppose that there is a surjective linear  mapping $A: \R^n \to \R^m$ such that $\calm (f-A) (\bar x) < \sur A$. Then $n\ge m$ and
$$
 \lopen f(\bar x) \geq \sur A - \calm (f-A) (\bar x) > 0.
$$
\end{theorem}

\begin{proof} Clearly, if $n<m$, there is no chance to have a linear surjection from $\R^n$ onto $\R^m$. Therefore $n\ge m$.
Without any loss of generality assume that $\bar{x}=0$ and $f(\bar x)=0$. Let us identify a~linear mapping $A$ with its matrix representation in the canonical bases of $\R^n$ and $\R^m$. Then $A \in \R^{m \times n}$ has a full rank $m$. Hence the (symmetric) matrix $AA^{T} \in \R^{m\times m}$ is non-singular. Let $B:= A^{T}(AA^{T})^{-1} \in \R^{n \times m}$.    Note that $\sur A$ is equal to the smallest singular value of  $A$ and  $\|B\|$ is equal to the largest singular value of $B$. As
$$
 B^TB = \big(A^{T}(AA^{T})^{-1} \big)^T A^{T}(AA^{T})^{-1} = \big((AA^{T})^{-1} \big)^T = \big((AA^{T})^{T} \big)^{-1} = (AA^T)^{-1},
$$
the singular values of $A$ and $B$ are reciprocal. Therefore $\|B\| = 1/\sur A$.
Pick any $c \in (0, \sur A - \calm (f-A)(0))$. Let $\gamma > 0$ be such that $\calm (f-A)(0) + c + \gamma  < \sur A$.  By
the
assumptions, there is   $\varepsilon > 0$ such that  $f$ is continuous on $\ball(0,2\varepsilon)$ and
\begin{equation} \label{eqfuLub}
 \|f(x) - Ax \| \leq (\calm (f-A)(0) + \gamma) \,  \|x\| \quad \mbox{whenever} \quad x \in \ball(0,2\varepsilon).
\end{equation}
Fix any $t \in (0, \varepsilon)$. Pick an arbitrary $y \in \ball[0,c t]$. Define the mapping $h_y: \ball_{\R^n}\to \R^m$ by
\begin{equation} \label{defhy}
 h_y(u) := \frac 1 t B\left(A(tu) - f(tu) + y\right), \quad u\in\ball_{\R^n}.
\end{equation}
Note that, for every $u \in \ball[0,2]$, we have $tu \in \ball(0,2\varepsilon)$.
In particular,  $h_y$ is well defined and continuous on $\ball_{\R^n}$. Given $u \in \ball_{\R^n}$, inequality \eqref{eqfuLub} with $x:=tu$ implies that
\begin{eqnarray*}
\|h_y(u) \| & \leq &  \frac1 t  \|B\| \, \big \|(A(tu) - f(tu)) +y\big\|\\
&\leq & \frac{\|B\|}{t} \big((\calm (f-A)(0) + \gamma) \, \|tu\| + \|y\|\big)  \leq \frac{\|B\|}{t}  \big((\calm (f-A)(0) + \gamma) t + c t \big) \\
&= &  \|B\| (\calm (f-A)(0) + c + \gamma)   <  \|B\| \ \sur A = 1.
\end{eqnarray*}
Therefore $h_y$ maps $ \ball_{\R^n}$ into itself. Using Brouwer's fixed point theorem, we find $u_y \in \ball_{\R^n}$ such that $h_y(u_y)=u_y$. Hence $Ah_y(u_y) = Au_y$. As $AB = I_{\R^m}$, the definition of $h_y$ implies that
$$
 A(tu_y) - f(tu_y) + y = tA(u_y) = A(tu_y).
$$
Then  $x_y:=tu_y$ is such that  $f(x_y)=y$ and $\|x_y\| \leq t$. Hence $y \in f(\ball[0,t])$. Since $y \in \ball[0,c t]$ was chosen arbitrarily, we have $\ball[0,c t] \subset  f(\ball[0,t])$. Therefore $\lopen f(\bar x) \geq c$. Letting $c \uparrow (\sur A - \calm (f-A)(0))$, we finish the proof.
\end{proof}

The above statement is quite similar to Theorem~\ref{Gr} with one important difference. If, in addition to the assumptions of Theorem~\ref{thmKaluza}, the mapping $A$ is invertible, then $n=m$ and $\sur A = 1/\|A^{-1}\|$. Consequently,
$$
 \lopen f(\bar x) \geq 1/\|A^{-1}\| - \calm (f-A) (\bar x).
$$
However, Example~\ref{ex01} shows that one cannot conclude that $f$ is {\it strongly} semiregular at $\bar x$, that is, that the mapping $f^{-1}$ has a single-valued localization around $(\bar x, f (\bar x))$.  This example also shows that we can have $\sur f(\bx) = 0$ although all the assumptions of Theorem~\ref{thmKaluza} hold.

We immediately obtain that the surjectivity of the Fr\'echet derivative at the reference point implies the openness with a linear rate of the mapping in question  at this point. The following result improves \cite[Corollary 1G.6]{book} where a weaker property of openness is shown. This statement was motivated by a discussion of the second author with V. Kalu\v{z}a, who suggested a proof using Borsuk-Ulam theorem.

\begin{corollary} \label{corKaluza}
Consider a point $\bar x \in \R^n$ along with a mapping $f : \R^n \to \R^m$ which is both defined and continuous in a vicinity of $\bar x$ and Fr\'echet differentiable at $\bar x$.  If $f'(\bar x)$ is surjective, then $n\geq m$ and
$ \lopen f(\bar x) \geq \sur f'(\bar x) > 0$.
\end{corollary}

We also obtain an extension of \cite[Theorem 1G.3]{book}.
\begin{theorem}
 Suppose  that the assumptions of Theorem~\ref{thmKaluza} hold and denote by $\varSigma$ the set of all selections for $f^{-1}$ defined in a vicinity of $\bar y:=f(\bar x)$.  Then
$$
 \inf_{\sigma \in \varSigma}  \calm \sigma  (\bar y) \leq \frac{1}{\sur A - \calm (f - A)(\bar x )}
$$
and
$$
   \inf_{\sigma \in \varSigma}  \calm (\sigma - A^{T}(AA^T)^{-1})(\bar y) \leq \frac{\calm (f - A)(\bar x )}{\sur  A \,(\sur A - \calm (f - A)(\bar x ))} \ .
 $$
 In particular, if $f$ is Fr\'echet differentiable at $\bar x$, then there is $\sigma \in \varSigma$ which is  Fr\'echet differentiable at $\bar y$ and
 $$
  \sigma'(\bar y) = [f'(\bar x)]^* (f'(\bar x) \, [f'(\bar x)]^*)^{-1}).
 $$
\end{theorem}
\begin{proof}
Let $B$, $c$, $\gamma$, $\varepsilon$, and  $t$ be as in the proof of Theorem~\ref{thmKaluza}. Consider the mapping
  $$
    V:= \ball[0,ct] \ni y \longmapsto \sigma(y):=x_y \in \ball[0,t]=:U,
  $$
  where $x_y$ is such that $h_y(x_y/t)=x_y/t$ with $h_y$ defined in \eqref{defhy}. We already know that $f(\sigma(y)) =y$ for each  $y \in V$.
	Moreover, given   $y \in V$, we have by \eqref{defhy} and \eqref{eqfuLub}
  \begin{eqnarray*}
     \|\sigma (y) \| &=& \|t h_y(\sigma(y)/t)\| = \|B\left(A(\sigma(y)) - f(\sigma(y)) + y\right)\| \\
    & \leq &  \|B\| \big( (\calm (f-A)(0) + \gamma)  \, \|\sigma (y)\| + \|y\|\big).
  \end{eqnarray*}
As $\|B\| = 1/ \sur A$ and  $\calm (f-A)(0) + \gamma < \sur A $, the above estimate implies that
\begin{equation} \label{eqsigma}
 \|\sigma (y) \| \leq \frac{1}{ \sur A - \calm (f-A)(0) - \gamma} \, \|y\| \quad \mbox{whenever} \quad  y \in V.
\end{equation}
Moreover, for a fixed $y \in V$, we have  by \eqref{defhy} and \eqref{eqfuLub}
 \begin{eqnarray*}
     \|\sigma (y) - By \| &=& \|t h_y(\sigma(y)/t) - By\| = \|B\left(A(\sigma(y)) - f(\sigma(y))\right)\| \\
     &\leq&  \|B\| (\calm (f-A)(0) + \gamma)  \,  \|\sigma (y)\|.
  \end{eqnarray*}
Using \eqref{eqsigma}, we get
\begin{equation} \label{eqsigmad}
 \|\sigma (y) - By\| \leq \frac{\calm (f-A)(0) + \gamma}{ \sur A \, (\sur A - \calm (f-A)(0) - \gamma)} \, \|y\| \quad \mbox{whenever} \quad  y \in V.
\end{equation}
As $\gamma > 0$ can be arbitrarily small, \eqref{eqsigma} and \eqref{eqsigmad}, respectively, imply the desired estimates.

To prove the second part, it suffices to observe that if  $f$ is Fr\'echet differentiable at $\bar x$ then $\calm (f - f'(\bar x))(\bar x )=0$.
\end{proof}

A similar approach as in the proof of Theorem~\ref{thmKaluza}, but applying Kakutani's fixed point theorem instead of Brouwer's theorem, yields a sufficient condition for openness with a linear rate of a set-valued mapping satisfying certain ``strong monotonicity/ellipticity" assumptions.

\begin{theorem} \label{thmKaluzeSetValued}
 Consider positive  constants $\ell$ and $r$, a point $(\bar x, \bar y) \in \R^n \times \R^n$,  and a mapping $F: \R^n \tto \R^n$ with $(\bar x, \bar y) \in \gph F$. Assume that $F$ has a closed graph and  convex values, the set $F(\ball[\bar x,r])$ is bounded, and that one of the following conditions holds:
\begin{itemize}
 \item[(C1)] for each $x \in \ball[\bar x,r]$ there is $y \in F(x)$ such that
$   \langle y - \bar y, x - \bar x \rangle \geq \ell \|x - \bar x \|^2$;
 \item[(C2)] for each $x \in \ball[\bar x,r]$ there is $y \in F(x)$ such that
$ \langle \bar y - y, x - \bar x \rangle \geq \ell \|x - \bar x \|^2$.
\end{itemize}
Then $  \lopen F(\bar x, \bar y) \geq \ell$; more precisely,
\begin{equation} \label{eqKaluzaSetValued}
 \ball[\bar y,\ell t] \subset F(\ball[\bar x,t]) \quad \mbox{for each}  \quad  t \in (0,r].
\end{equation}
\end{theorem}

\begin{proof}
Note that  \eqref{eqKaluzaSetValued} for $F$ satisfying (C2) follows by considering the reference point $(\bar x,-\bar y)$ and the mapping $-F$, which necessarily satisfies (C1). Suppose that (C1) holds.  Assume without any loss of generality that $(\bar x, \bar y) = (0,0)$.    Find  $m > 0$ such that
 $F(\ball[0,r]) \subset \ball[0,m]$.

 First, we show that
\begin{equation} \label{eqKaluzaSetValuedb}
 \ball[0,ct] \subset F(\ball[0,t]) \quad \mbox{for each} \quad  c \in (0,\ell) \quad \mbox{and each} \quad  t \in (0,r].
\end{equation}
Let $c$ and $t$ be as in \eqref{eqKaluzaSetValuedb}. Fix an arbitrary (non-zero) $y \in \ball[0,c t]$. Pick $\alpha > 0$ such that
$$
 2\alpha \ell < 1 \quad \mbox{and} \quad  \alpha (m + cr)^2 < 2(\ell - c) t^2.
$$
 Define the mapping $H:
\ball[0,t]\tto\ball[0,t]
$, depending on the choice of $(y, c, t, \alpha)$,  by
$$
 H(u) := \big(u + \alpha (y - F(u)) \big) \cap \ball[0,t], \quad u\in\ball[0,t].
$$
Fix any $u \in \ball[0,t]$.   Using (C1), we find  a point $v \in F(u)$ such that $ \langle v, u  \rangle \geq  \ell \|u  \|^2  $. Let $z:= u + \alpha (y - v)$. Then
\begin{eqnarray*}
 \|z\|^2 & =  & \|u\|^2 + 2 \alpha \langle u, y - v \rangle  + \alpha^2 \|y - v \|^2 =  \|u\|^2 - 2 \alpha \langle v, u \rangle  + 2 \alpha \langle u, y \rangle +   \alpha^2 \|y - v \|^2 \\
 &\leq &  (1 - 2 \alpha \ell ) \|u\|^2   +  2\alpha \|u\| \|y\| + \alpha^2 (\|v\| + \|y\|)^2 \\
 &\leq & (1 - 2\alpha \ell) t^2 + 2 \alpha  t (c t) + \alpha^2 (m + cr)^2 <  \big(1 + 2\alpha (c - \ell)\big) t^2 + 2\alpha (\ell - c)t^2 = t^2.
\end{eqnarray*}
Hence $z \in H(u)$.  Consequently,  the domain of $H$ is equal to $\ball[0,t]$, which is a non-empty compact convex set. Since $F$ has closed graph and convex values, we conclude that $H$ has the same properties. Applying Kakutani's fixed point theorem, we find $u \in \ball[0,t]$ such that $u \in H(u)$. This implies that $y \in F(u) \subset F( \ball[0,t])$. As $y \in \ball[0,ct]$, and also $(c,t) \in (0,\ell) \times (0,r]$
are
arbitrary, \eqref{eqKaluzaSetValuedb} is proved.

To show \eqref{eqKaluzaSetValued}, fix any $t \in (0,r]$. Pick an arbitrary $y \in  \ball[0,\ell t]$. Let $y_k := (1-1/k)y$ for each $k \in \N$. Then $(y_k)$ converges to $y$.  For each $k \geq 2$, using \eqref{eqKaluzaSetValuedb} with $c:=(1-1/k)\ell$, we find $x_k \in \R^n$ such that $y_k \in F(x_k)$ and $\|x_k\| \leq t$. Passing to a subsequence, if necessary, we may assume that $(x_k)$ converges to, say, $ x \in \R^n$.
Then $\| x\| \leq t$ and $y \in F (x)$ because $\gph F$ is closed.  So  $F( \ball[0,t])$ contains $y$, which
is
an arbitrary point in $\ball[0,\ell t]$.
\end{proof}

The above statement implies \cite[Theorem 1 and Corollary 1]{BR} under slightly weaker assumptions and the above proof also shows that there is no need to extend the locally defined mapping under consideration on the whole space.

\begin{corollary} \label{corKaluzaSetValued1}
 Consider positive  constants $\ell$ and $r$, a point $\bar x \in \R^n$,  and a mapping $F: \R^n \tto \R^n$ with $\dom F=\ball[\bar x, r]$. Assume that $F$ is upper semicontinuous, has compact convex values, and
\begin{equation} \label{eqROSLw}
 \forall \ x \in \ball[\bar x,r] \ \forall \bar y \in F(\bar x) \ \exists y \in F(x) :
   \langle \bar y - y, x - \bar x \rangle \geq \ell \|x - \bar x \|^2.
\end{equation}
Then, for each $y \in \R^n$ such that $\dist (y, F(\bar x)) \leq r\ell$, there is $x\in \ball[\bar x,r]$ satisfying
$$
  y \in F(x) \quad \mbox{and} \quad \|x - \bar x\| \leq \frac{1}{\ell} \dist \big(y, F(\bar x)\big).
$$
\end{corollary}
\begin{proof} Since $F$ is upper semicontinuous and has compact values, using a standard compactness argument we conclude that the set $F(\ball[\bar x, r])$ is bounded. Moreover, $\gph F$ is closed since $F$ is upper semicontinuous with closed values,  closed domain, and
bounded range. Fix any $y \in \R^n$ with $ r\ell \geq \dist (y, F(\bar x))$ ($>0$). As  $F(\bar x)$ is a compact set, there is $\bar{y} \in \R^n$ such that $\|y - \bar{y}\| = \dist (y, F(\bar x))$. Now \eqref{eqROSLw} implies that (C2) is satisfied. By \eqref{eqKaluzaSetValued} with $t:= \|y - \bar{y}\|/\ell \leq r$,  there is $x \in \ball[\bar x, \|y - \bar{y}\|/\ell]=  \ball[\bar x, \dist (y, F(\bar x))/\ell] \subset \ball[\bar x,r]$ such that $y \in F(x).$
\end{proof}

We also get:
\begin{corollary} \label{corKaluzaSetValued2}
 Consider positive  constants $\ell$ and $r$, a point $\bar x \in \R^n$,  and a mapping $F: \R^n \tto \R^n$ with $\dom F=\ball[\bar x, 2r]$. Assume that $F$ is upper semicontinuous, has compact convex values, and
\begin{equation} \label{eqROSL}
 \forall  x, x' \in \ball[\bar x,2r] \ \forall  y \in F(x) \ \exists y' \in F(x') :
   \langle  y - y', x' - x \rangle \geq \ell \|x' - x \|^2.
\end{equation}
Then $  \sur F(\bar x, \bar y) \geq \ell$; more precisely,
\begin{equation} \label{eqKaluzaSetValuedaround}
 \ball[ y,\ell t] \subset F(\ball[x,t]) \quad \mbox{whenever}  \quad  (x,y) \in (\ball[\bar x, r] \times \ball[\bar y , r])\cap \gph F  \mbox{ and }  t \in (0,r].
\end{equation}
\end{corollary}
\begin{proof}
Fix any $(x,y)$ and $t$ as in \eqref{eqKaluzaSetValuedaround}. Then $\ball[x, r] \subset \ball[\bar x, 2r]$. Hence, \eqref{eqROSL} implies that for each $x' \in \ball[x,r]$ there is $y' \in F(x')$ such that $\langle  y - y', x' - x \rangle \geq \ell \|x' - x \|^2$, which is (C2) with $(\bar x, \bar y, x,y)$ replaced by $(x,y,x',y')$. As in the proof of Corollary~\ref{corKaluzaSetValued1}, we conclude that all the assumptions of Theorem~\ref{thmKaluzeSetValued} with $(\bar x,\bar y):=(x,y)$ are satisfied.
\end{proof}

\begin{remark} \label{rmSVKAL} \rm
Given $\ell > 0$, condition  \eqref{eqROSLw} holds, in particular, if $F$ is relaxed one-sided Lipschitz (ROSL) on $\ball[\bar{x},r]$ with the constant $-\ell$ in the sense of  \cite[Definition 1]{BR}, that is,
$$
 \forall  x, x' \in \ball[\bar x,r] \ \forall  y \in F(x) \ \exists y' \in F(x') :
   \langle  y - y', x - x' \rangle \leq  -\ell \|x - x' \|^2.
$$
Condition \eqref{eqROSL} means that  $F$ is ROSL on $\ball[\bar{x},2r]$ with the constant $-\ell$.  Up to minor changes in notation,   Corollary~\ref{corKaluzaSetValued2} seems to be the statement which the authors tried to formulate and prove in  \cite[Corollary 2 (ii)]{BR} under an additional assumption that $F$ is (Hausdorff) continuous.
However, their formulation seems  to be not
completely correct, since (local) metric regularity at $(\bar x, \bar y)$
presumes the reference point  \emph{to lie in} $  \gph F$. So the assumption in \cite[Corollary 2 (ii)]{BR} that $\dist(\bar y, F(\bar{x}))$ is small enough holds trivially. Also note that  ``a slightly generalized definition of metric regularity" in \cite{BR} is nothing else but the \emph{usual} definition of this property because $F: \R^n \tto \R^n$ in \cite{book}  means neither that $\dom F = \R^n$ nor that $\bar x$ is an interior point of $\dom F$.
\end{remark}

\begin{remark} \rm
  A sufficient condition for semiregularity of  a continuous (possibly nonsmooth) mapping $f:\R^n \to \R^n$ by using equi-invertibility of a pseudo-Jacobian can be found in  \cite[Theorem 3.2.1]{JLbook}.
\end{remark}

\section{General conditions and semiregularity of the sum} \label{secSum}

First, we present sufficient as well as necessary conditions for semiregularity of a~single-valued mapping.
\begin{proposition}\label{propSoft} Let $(X,d)$ be a complete metric space and $(Y,\varrho)$ be a metric space. Consider a~point $\bar{x} \in X$, a~continuous mapping $g:X\to Y$,  whose domain is all of $X$, and positive constants $c$ and $r$.
\begin{itemize}
\item[(i)] Assume that for every $x \in \ball(\bar{x},r)$ and every $y \in \ball(g(\bar{x}),cr)$  satisfying
	\begin{equation} \label{eqpremise2}
	0 < \varrho(g(x),y) \leq  \varrho(g(\bar{x}),y) - c \, d(x,\bar{x})
	\end{equation}
	there is a point  $x'\in X$ such that
	$$
	\varrho(g(x'),y) < \varrho(g(x),y)- c \, d(x,x').
	$$
	Then  $g\big(\ball(\bar{x},t)\big) \supset \ball(g(\bar{x}),ct)$ for every $t\in(0, r)$.
\item[(ii)] Assume  that  $ g\big(\ball(\bar{x},t)\big) \supset \ball(g(\bar{x}),ct)$  for every $t\in(0, r)$.
Then for every $c' \in (0,c)$,  every $x \in \ball(\bar{x},r)$, and every $y \in  \ball(g(\bar{x}),c'r)$  satisfying
\begin{equation} \label{eqpremise2c}
  0 < \varrho(g(\bar{x}),y) \leq  \varrho(g(x),y) - c' \, d(x,\bar{x})
\end{equation}
there is a point  $x'\in X$ such that
$$
 \varrho(g(x'),y) < \varrho(g(x),y)-c' \, d(x,x').
$$	
\end{itemize}
\end{proposition}

\begin{proof} (i)
	Fix any $t \in (0,r)$. Pick an arbitrary $y\in \ball(g(\bar{x}),ct)$. We shall find a~point $u \in  \ball(\bar{x},t)$ such that $g(u)=y$. If $y= g(\bar{x})$, then we set $u:= \bar{x}$. Assume that $y \neq g(\bar{x})$.  Define a~(continuous) function $f:X\to [0,\infty)$ by $f(x):=\varrho(g(x),y)$, $x \in X$.
	Then  $f(\bar{x})=\varrho(g(\bar{x}), y) < c t \  (< \infty)$.
	Employing the Ekeland variational principle, we find a point $u \in X$ such that
	\begin{equation} \label{eqek1_1}
	\varrho(g(u), y)  \le \varrho(g(\bar{x}), y) -c \, d(u,\bar x)
	\end{equation}
and
	\begin{equation} \label{eqek1}
	  \varrho(g(v),y)   \geq  \varrho(g(u),y) - c \,  d(v,u) \quad \mbox{for every} \quad v \in X.
	\end{equation}
	By \eqref{eqek1_1}, we have that  $c \, d (u,\bar{x}) \leq \varrho(g(\bar{x}), y) < ct$. Hence $u \in \ball(\bar{x},t)$. We claim  that  $g(u)=y$. Assume, on the contrary, that $g(u )\neq y$.
	As $t < r$,  we have $ u \in \ball(\bar{x},r)$ and $y \in  \ball(g(\bar{x}),cr)$. Then \eqref{eqek1_1} implies that \eqref{eqpremise2} with $x:=u$ holds. Find a point $x'\in X$ such that $	 \varrho(g(x'), y) < \varrho(g(u), y) - c \, d (u ,x')$.
	Setting	$v:=x'$ in  \eqref{eqek1}, we get that
	$\varrho(g(x'),y) \geq    \varrho(g(u ),y) - c \, d(u ,x')$,
	a~contradiction. Consequently $y = g(u)$ as claimed, and so $y\in g\big(\ball(\bar{x},t)\big)$. Since $y \in \ball(g(\bar{x}),c t)$ is arbitrary, the proof is finished.
	
	(ii)  Fix any $c' \in (0,c)$, any $x \in  \ball(\bar{x},r)$, and any $ y \in \ball(g(\bar{x}),c'r)$  satisfying \eqref{eqpremise2c}. Let
 $t := \varrho(g(\bar{x}),y)/c'$. The choice of $y$ implies that $0   <  t < (c'r)/c' = r$.
 As $y \in \ball[g(\bar{x}),c't] \subset \ball(g(\bar{x}),ct) $ there is $x' \in \ball(\bar{x},t)$ such that $g(x')=y$.
 Then
 $$
  c' \, d(x,x') \overset{(\bigtriangleup)}{<}   c' \, d(x,\bar{x}) + c' t \overset{\eqref{eqpremise2c}}{\leq}
  \varrho(g(x),y) -  \varrho(g(\bar{x}),y) + c't  =  \varrho(g(x),y)
   =  \varrho(g(x),y) - \varrho(g(x'),y).
$$
\end{proof}

Although the above statement concerns single-valued mappings, using the restriction of the canonical projection to the graph of a given set-valued mapping we immediately get its set-valued version. Moreover, it can be directly used to establish semiregularity of the sum of two set-valued mappings -  Theorem~\ref{thmCFK} below.

\begin{proposition} \label{propSoftSet} Let $(X,d)$ and  $(Y,\varrho)$ be metric spaces and a point $(\bar{x}, \bar{y}) \in X \times Y$ be given. Consider a set-valued mapping $F: X\rightrightarrows Y$, with $ \bar{y}  \in  F (\bar{x})$, for which there are positive constants  $c$, $r$, and $\alpha$ such that $\alpha c < 1$ and that the set $ \mbox{\rm gph} \, F \cap \big(\ball[\bar{x}, r] \times \ball[\bar{y}, r/\alpha]  \big)  $ is complete.
\begin{itemize}
\item[(i)]
   Assume that for every $x \in \ball(\bar{x}, r)$, every $v \in  \ball(\bar{y}, r/\alpha) \cap F(x)$, and every $y \in \ball(\bar{y}, cr)$  satisfying
	\begin{equation} \label{eqpremiseSETVALUED1}
\quad 0 < \varrho(v,y) \leq  \varrho(\by,y)-	c \, \max\{ d(x,\bar{x}), \alpha \varrho(v,\by)\}
	\end{equation}
	there is a pair $(x',v') \in \mbox{\rm gph} \, F$ such that
	\begin{eqnarray}\label{HvezdaS2}
	 \varrho(v',y)  < \varrho(v,y)- c \max\{d(x,x'), \alpha \varrho(v,v') \}.
	\end{eqnarray}
	Then  $F(\ball(\bx,t))\supset \ball(\by,c t)$ for every $t\in (0,r)$.
\item[(ii)] Assume that $F\big(\ball(\bar{x},t)\big) \supset \ball(\by,ct)$ for every $t\in(0, r)$.
	Then for every $c' \in (0,c)$,  every $x \in \ball(\bar{x}, r)$, every $v \in \ball(\bar{y}, r/\alpha) \cap F(x)$,  and every $y \in \ball(\bar{y}, c'r)$  satisfying
	\begin{equation} \label{eqpremise2cSet}
	0 < \varrho(\by,y) \leq  \varrho(v,y) - c' \, \max\lbrace d(x,\bar{x}),\alpha \varrho(v,\by)\rbrace
	\end{equation}
		there is a pair $(x',v') \in \mbox{\rm gph} \, F$ such that
	\begin{eqnarray*}
	\varrho(v',y)  < \varrho(v,y)- c' \max\{d(x,x'), \alpha \varrho(v,v') \}.
	\end{eqnarray*}
\end{itemize}
\end{proposition}

\begin{proof}  (i) Define the (compatible) metric $\tilde{d}$ on the space $X \times Y$ for  each $(u,w)$, $(u',w') \in X \times Y$ by $\tilde{d} \big((u,w),(u',w') \big):= \max\{d(u,u'), \alpha \varrho(w,w') \}$. Then  $\widetilde{X} := \big(\ball[\bar{x}, r] \times \ball[\bar{y}, r/\alpha]  \big) \cap \mbox{\rm gph} \, F  $, equipped with  $\tilde{d}$, is a complete metric space. Let $g:\widetilde{X} \to Y$ be defined by $g(x,y)=y$, $(x,y) \in  \widetilde{X}$. Then $g$ is a~continuous mapping defined on the whole $\widetilde{X}$.  Fix any $(x,v) \in \ball_{\widetilde{X}}((\bar{x}, \bar{y}),r) = \big(\ball(\bar{x}, {r}) \times \ball(\bar{y}, {r}/\alpha) \big) \cap \mbox{\rm gph} \, F  \subset  \widetilde{X} $ and any $y \in \ball(\bar{y}, c r )$ such that
	\eqref{eqpremiseSETVALUED1} holds.   Find a pair $(x',v') \in \mbox{\rm gph} \, F$   satisfying \eqref{HvezdaS2}. Then
	 \begin{eqnarray*}
		\tilde{d} \big( (x',v'), (\bar{x}, \bar{y})\big) & \leq & \tilde{d} \big( (x',v'), (x, v)\big) +  \tilde{d} \big( (\bx,\by), (x, v)\big)
		 \\ &\overset{\eqref{HvezdaS2},\eqref{eqpremiseSETVALUED1}}{<}& \frac{\varrho(v,y)- \varrho(v',y)}{c} + \dfrac{\varrho(\by,y)-\varrho(v,y)}{c}=\frac{\varrho(\by,y)- \varrho(v',y)}{c}
		<  \frac{cr}{c} =  r.
			\end{eqnarray*}
		Hence,  $(x',v') \in \widetilde{X}$.   Proposition~\ref{propSoft}, with $(X,d,\bx):=(\widetilde{X}, \tilde{d}, (\bx,\by))$,  implies that
		$$
		\ball(\by,ct) \subset g\Big( \big(\ball(\bx, t) \times \ball(\by, t/\alpha)\big) \cap   \gph \, F  \Big)
		\quad \text{for each} \quad t \in (0,  r).
		$$
		Fix an arbitrary $t \in (0,r)$. Given  $y \in \ball(\by,ct)$, there are $x \in \ball(\bx,t)$ and $y' \in  \ball(\by, t/\alpha ) \cap F(x) $ such that $g(x,y')=y$, hence $y'=y$ and consequently  $y \in F(x)$. Thus $\ball(\by,ct) \subset F\big(\ball(\bx,t)\big)$.
		
		(ii) Fix any $c' \in (0,c)$, then fix any $(x,v) \in \gph \, F \cap \big(\ball(\bx, r) \times \ball(\by, r/\alpha) \big)$ and any $y\in\ball(\by,c'r)$ satisfying \eqref{eqpremise2cSet}. Let
	$t := \varrho(\by,y)/c'$. The choice of $y$ implies that $0   <  t < (c'r)/c' = r$.
	As $y \in \ball[\by,c't] \subset \ball(\by,ct) $ there is $x' \in \ball(\bar{x},t)$ such that $F(x')\ni y$.
	Let $v':=y$. Then
	$$
	c' \, d(x,x') \overset{(\bigtriangleup)}{<}  c' \, d(x,\bar{x}) + c' t \overset{\eqref{eqpremise2cSet}}{\leq}
	\varrho(v,y) -  \varrho(\by,y) + c't  =  \varrho(v,y)
	=  \varrho(v,y) - \varrho(v',y)
	$$
	and
	$$
	c'\alpha \varrho(v,v')<\varrho(v,v')=\varrho(v,y)-\varrho(v',y).
	$$
\end{proof}

The next example shows that the assumptions of Proposition~\ref{propSoftSet}(i) do not imply that the mapping under consideration is regular around the reference point and can provide a tight lower estimate for
the corresponding modulus.
\begin{example} \label{exSur} \rm
	Let $(X,d):= (Y,\varrho) := (\mathbb{R},\vert \cdot\vert)$ and $(\bx,\by):=(0,0)$. Consider a set-valued mapping $\mathbb{R}\ni x \longmapsto F(x) := \lbrace x,0\rbrace\subset \mathbb{R}$. Then $F$ has a closed graph and $\sur F(0,0)=0$ while $\lopen F(0,0)=1$.   Fix any $c\in (0,1)$, any $\alpha\in (0,1/c)$, and any $r>0$. Pick any  $x \in  \ball(0,r)$, any $v \in  \ball(0,r/\alpha)\cap F(x)$, and  any $y\in \ball(0,cr)$ satisfying
	\begin{eqnarray}
	\label{exEstimate}
	0<\vert v-y\vert\leq \vert y\vert -c \max\lbrace\vert x\vert, \alpha \vert v\vert \rbrace.
	\end{eqnarray}
	Let  $(x',v'):=(y,y) \in \gph F$.  If $v\neq 0$ then  $v=x$ and consequently $x\neq y$ by \eqref{exEstimate}. Hence
	$
	c \, \vert x-x'\vert<\vert x-x'\vert =\vert v -y\vert =\vert v -y \vert-\vert v'-y\vert
$	and
$	\alpha c\vert v-v'\vert<\vert v-v'\vert=\vert v -y\vert =\vert v -y\vert-\vert v'-y\vert$.  If $v=0$ then \eqref{exEstimate} implies that  $y\neq 0$ and $x=0$. Thus
$
c \, \vert x - x'|= c\vert x'\vert<\vert x'\vert =\vert y \vert =\vert y\vert-\vert v'-y\vert
$
and
$
\alpha c\vert v- v'\vert = \alpha c\vert v'\vert<\vert v'\vert=\vert y\vert =\vert y\vert-\vert v'-y\vert
$. In both the cases, we showed that
$$
c\max\lbrace\vert x-x'\vert,\alpha\vert v-v' \vert\rbrace<\vert v-y \vert -\vert v'-y\vert.
$$
Proposition~\ref{propSoftSet} implies that $\lopen \,F(0,0) \geq c$ for any $c \in (0,1)$.
\end{example}

\begin{theorem} \label{thmCFK}
	Let $(X,d)$ be a complete metric space, $(Y, \varrho)$ be a complete linear  metric space with a~shift-invariant metric, and a point $(\bx,\by, \bz)\in X\times Y\times Y$ be given. Consider  set-valued mappings $F$, $G:X\rightrightarrows Y$, with $(\by,\bz) \in F(\bx)\times G(\bx)$,  for which there are positive constants $c'$, $r$, and $\ell < c'$   such that both the sets $\gph F\cap \big(\ball[\bx,2r]\times\ball[\by,2c'r]\big)$ and $\gph G\cap \big(\ball[\bx, 2r]\times\ball[\bz,2\ell r]\big)$ are closed; that
	\begin{equation}
	\ball(v,c' \tau) \subset F(\ball(x,\tau)) \quad \mbox{whenever} \qquad 	x \in \ball(\bx,r), \ v \in F(x)\cap \ball(\by,c' r), \ \mbox{and} \quad  \tau\in (0,r );  \label{LinOpennes1}
	\end{equation}
	and that
	\begin{eqnarray} \label{Aubin}
	   G(x)\cap\ball(\bz,\ell r )\subset G(x')+ \ell \, d(x,x')\ball_Y\quad\text{for each}\quad x,x'\in \ball(\bx,2r).
	\end{eqnarray}
	Then
	\begin{equation} \label{eqLinOpen}
       (F+G)\big(\ball(\bx,t)\big) \supset \ball(\by+\bz,(c-\ell)t)      \quad \text{whenever} \quad c \in (\ell, c') \quad
       \mbox{and} \quad t \in (0,r).  	
	\end{equation}
   \end{theorem}
\begin{proof} Fix any $c \in (\ell, c')$.
    Define the (compatible) metric on $X\times Y\times Y$ by
	\begin{eqnarray*}
		\tilde{d}((x,v,z),(x',v',z')):=\max\lbrace d(x,x'), \varrho(v,v')/c',\varrho(z,z')/\ell \rbrace,\quad  (x,v,z),(x',v',z')\in X\times Y\times Y.
	\end{eqnarray*}
	  Let
	\begin{eqnarray*}
	\widetilde{X}:=\lbrace (x,v,z)\in X\times Y\times Y: \ x\in \ball[\bx,2r], v\in F(x)\cap \ball[\by,2c'r], z\in G(x)\cap\ball[\bz,2\ell r]\rbrace.
	 \end{eqnarray*}
	Then $\widetilde{X}$ is a (nonempty) closed subset of $X \times Y \times Y$, hence $(\widetilde{X},\tilde{d})$ is a complete metric space.
	Let  $g: \widetilde{X}\to Y$ be defined by
	\begin{eqnarray*}
		g(x,v,z):=v+z,\quad  (x,v,z)\in\widetilde{X}.
	\end{eqnarray*}
	Then $g$ is a~continuous mapping defined on the whole $\widetilde{X}$.  Fix any $(x,v,z)\in \ball_{\widetilde{X}}\big((\bx,\by,\bz),r\big) \subset \ball(\bx,r)\times \ball(\by,c'r) \times \ball(\bz,\ell r)$ and any $y\in \ball(\by+\bz,(c-\ell) r)$ such that
	\begin{eqnarray*}
		0<\varrho(g(x,v,z),y)\leq \varrho(\by+\bz,y)-(c-\ell)\tilde{d}((x,v,z),(\bx,\by,\bz)).
	\end{eqnarray*}
	Let $ \tau:=\varrho(g(x,v,z),y)/c$. Then $0 < \tau \leq \varrho(\by+\bz,y)/c < (c-\ell) r/c < r$.
	As 	
    $$
	\varrho(y-z,v)=\varrho(y-z,g(x,v,z)-z)=\varrho(y,g(x,v,z))=c\tau<c'\tau,
	$$
	we have $v':=y-z\in \ball(v,c'\tau)$.	
	By \eqref{LinOpennes1}, there is $x'\in \ball(x,\tau)$ such that $v'\in F(x')$. Then
	\begin{equation}
	\label{xEstimate}
	d(x',\bx) \leq d(x',x)+d(x,\bx)< \tau+r<  2r.
	\end{equation}
	Since $z\in G(x)\cap \ball(\bz,\ell r)$, using \eqref{Aubin}, we find $z'\in G(x')$ such that
	$ \varrho(z,z')\leq \ell d(x,x')< \ell \tau$.
    Then
    $$
      \varrho(z', \bar{z}) \leq \varrho(z',z) + \varrho(z, \bar{z}) < \ell \tau + \ell r < 2 \ell r
	\quad \mbox{ and } \quad
	 \varrho(v',\by) \leq 	\varrho(v',v)+\varrho(v,\by) < c\tau + c'r< 2c'r,
	$$
	hence, remembering \eqref{xEstimate}, we conclude that  $(x',v',z')\in\widetilde{X}$.  	
	Moreover,
	\begin{eqnarray*}
		\varrho(g(x',v',z'),y)&=&\varrho(y-z+z',y)=\varrho(z',z)< \ell \tau=c\tau-(c-\ell )\tau= \varrho(g(x,v,z),y)-(c-\ell)\tau.
	\end{eqnarray*}
	 Since $d(x',x) < \tau$, $\varrho(v',v) < c' \tau$, and  $\varrho(z',z)< \ell \tau$, we get that
	\begin{eqnarray*}
		\varrho(g(x',v',z'),y)<\varrho(g(x,v,z),y)-(c-\ell)\tilde{d}((x,v,z),(x',v',z')).
	\end{eqnarray*}
	Proposition \ref{propSoft}, with $(\widetilde{X},\tilde{d}, (\bx,\by,\bz), c-\ell)$ instead of  $(X,d,\bx, c)$, implies that
	$$
	 	g\big(\ball_{\widetilde{X}}((\bx,\by,\bz),t)\big) \supset  \ball(g(\bx,\by,\bz),(c-\ell)t) = \ball(\by+\bz,(c-\ell)t)         \quad \mbox{for each} \quad  t\in (0,r).
     $$
	Consequently, given  $t\in(0,r)$ and $y\in\ball(\by+\bz,(c-\ell)t)$, there are $x\in\ball(\bx,t)$, $v\in  F(x) \cap \ball(\by, c' t) $, and $z\in G(x)\cap \ball(\bz,  \ell t)$ such that $ y =  v + z$, that is, $y \in F(x)+G(x) 	\subset (F+G)(\ball(\bx,t))$.
\end{proof}

Using the above statement we immediately get the following result:
\begin{theorem} \label{thmCFKloc}
	Let $(X,d)$ be a complete metric space, $(Y, \varrho)$ be a complete linear  metric space with a~shift-invariant metric, and a point $(\bx,\by,\bz)\in X\times Y\times Y$ be given. Consider  set-valued mappings $F$, $G:X\rightrightarrows Y$ such that $F$ has a locally closed graph around $(\bx,\by) \in \gph F$ and $G$ has a locally closed graph around  $(\bx,\bz) \in \gph G$.
	Then
	\begin{equation} \label{Estimated}
	\lopen (F+G)(\bx,\by+\bz)\geq \sur F(\bx,\by)- \lip G(\bx,\bz).
	\end{equation}
\end{theorem}
\begin{proof}
	If $\sur F(\bx,\by)\leq \lip G(\bx,\bz)$ we are done. Suppose that $\sur F(\bx,\by)> \lip G(\bx,\bz)$. Fix any $c, c',$ and  $\ell$ such that $\lip G(\bx,\bz)<\ell <c<c'<\sur F(\bx,\by)$. Using the definitions, we find a~small enough $r > 0$ such that all the assumptions of Theorem~\ref{thmCFK} are satisfied. By \eqref{eqLinOpen}, we get $\lopen\big(F+G\big)(\bx,\by+\bz)\geq c-\ell$. Letting $\ell\downarrow \lip G(\bx,\bz)$ and $c\uparrow \sur  F(\bx,\by)$, we get \eqref{Estimated}.
\end{proof}

\begin{remark} \label{rmkCFK} \rm

\item[1.] It is well known, that one cannot replace  $\lopen (F+G)(\bx,\by+\bz)$ by $\sur (F+G)(\bx,\by+\bz)$ in \eqref{Estimated}, as the following elementary example shows (see also \cite[Example 5I.1]{book} for a more elaborate one). Let $(X,d):= (Y,\varrho) := (\mathbb{R},\vert \cdot\vert)$ and $(\bx,\by,\bz):=(0,0,0)$. Consider set-valued mappings $\mathbb{R}\ni x \longmapsto F(x) := \lbrace x,-1 \rbrace\subset \mathbb{R}$ and $\mathbb{R}\ni x \longmapsto G(x) := \lbrace 0,1 \rbrace\subset \mathbb{R}$. Then $F$ and $G$ have closed graphs,  $\sur F(0,0)=1$, and $\lip G(0,0)=0$. Then
$(F+G)(x)=\{x,-1,x + 1 ,0\}$ for each $x \in \mathbb{R}$. Consequently, $\lopen (F+G)(0,0) = 1$ while $\sur (F+G)(0,0) = 0$.
\item[2.] 	Suppose that the assumptions of Theorem~\ref{thmCFK} hold and let $(\tilde{x}, \tilde{y}, \tilde{z})  \in \ball(\bx,r/2) \times  \ball(\by,c'r/2) \times  \ball(\bz,\ell r/2)$ with $(\tilde{y}, \tilde{z}) \in F(\tilde{x})\times  G(\tilde{x})$ be arbitrary. Defining $\widetilde{X}$, $\tilde{d}$, and $g$ as in the proof of Theorem~\ref{thmCFK} and replacing $(\bx,\by,\bz,r)$  by  $(\tilde{x},
\tilde{y},\tilde{z}, r/2)$  in the rest of the proof, we get that
	\begin{equation} \label{eqLinOpenb}
      (F+G)(\ball(\tilde{x},t)) \supset \ball(\tilde{y}+\tilde{z},(c-\ell)t)      \quad \mbox{whenever} \quad c \in (\ell, c') \quad        \mbox{and} \quad t \in (0,r/2).  	
	\end{equation}
 Employing this technique, we get short proofs of the results in \cite{HNT2014}.
 Note that \eqref{eqLinOpenb} does not mean that $\sur (F+G)(\bx,\by+\bz)\geq c - \ell$ since, given $(\tilde{x},\widetilde{w}) \in \gph (F+G)$ close to $(\bx, \by + \bz)$, there is no guarantee that $\widetilde{w} = \tilde y + \tilde z$ for some pair $(\tilde{y}, \tilde{z})$ with the properties required above unless the so-called sum stability holds, cf. \cite{HNT2014}.

\end{remark}

To conclude this section, we present a closely related result which was published in \cite{ACN}.

\begin{theorem}  \label{thmSumr}
Let $(X, \|\cdot\|)$ and $(Y,\|\cdot\|)$ be Banach spaces and a point $(\bar x,  \bar y, \bar z) \in X \times Y \times Y$ be given. Consider set-valued mappings $F$, $G:X\rightrightarrows Y$, with    $(\bar{y},\bar{z})\in F(\bar x) \times G(\bar{x})$, for which there are positive constants  $a$, $b$, $\kappa$, and $\ell$ such that
$\kappa \ell < 1$; that both the sets $\gph F \cap \big(\ball[\bar x,a]\times \ball[\bar y,2a]\big)$ and
$\gph G \cap \big(\ball[\bar x, a] \times \ball[\bar z, 2a]\big)$ are closed; that
\begin{equation}
\label{defMRSUMr}
\dis \big(x, F^{-1}(y) \big) \leq \kappa \dis\big(y, F(x)\big)
\quad
\mbox{for each}
\quad  (x,y) \in \ball(\bar x,a)\times \ball(\bar y,a);
\end{equation}
and that
\begin{equation} \label{eqExcessr}
G(x)\cap\ball(\bar{z}, a) \subset G(x') + \ell \| x - x' \| \ball_Y \quad
\mbox{for each}
\quad x, x' \in \ball(\bar{x}, a).
\end{equation}
Then, for
any
$\beta >0$ such that  $2\beta  \max\{1, \kappa\} < a(1-\kappa \ell)$,
we have
\begin{equation}\label{ineqr}
 \dist\big(\bar x, ({F}+G)^{-1}(y) \big)\le \frac\kappa{1-\kappa \ell} \dist\big(y, {F}(\bar x)+\bar z\big)
\quad\mbox{for
each}
\quad y\in \ball(\bar y+\bar z,\beta).
\end{equation}
\end{theorem}

Note that the property in \eqref{ineqr} is  stronger than semiregularity in general. For the Newton-type methods (cf. Section~\ref{secNewton}), the semiregularity is enough and seems to play the key role in the analysis.

\section{Convergence of the Newton-type methods} \label{secNewton}

In this section, we  study inexact iterative methods of Newton type for solving  the generalized equation \eqref{GE}.  We focus on a local convergence analysis of  \eqref{niN} around a reference solution.

\begin{theorem} \label{mrr2}  Let $(X, \|\cdot\|)$ and $(Y, \|\cdot\|)$ be Banach spaces. Consider  a point $\bx \in X$ along with a continuous mapping $f:X \to Y$ and a set-valued mapping $F:X \tto Y$ with closed graph  such that $f(\bx) + F(\bx) \ni 0$.
Suppose  that there is $\mathcal{H}:X \tto {\cal L}(X, Y)$ which is upper semicontinuous at $\bx \in\inT \dom \mathcal{H}$ with
$\chi(\mathcal{H}(\bar{x})) < \infty$, and such that,  for each $A \in \mathcal{H}(\bx)$, the mapping $G_A:X \tto Y$ defined by
  \begin{equation} \label{G_A}
     G_A(x) := f(\bx) + A(x-\bx) + F(x), \quad x \in X,
  \end{equation} is regular around $(\bx,0)$, and
 \begin{equation} \label{eqH2}
   \lim_{x \to \bx, \ x \neq \bx} \frac{\sup_{A\in \mathcal{H}(x)} \| f(x)-f(\bx) - A(x-\bx)\| }{\|x - \bx\|} = 0.
 \end{equation}
Let  $(R_k)$ be a sequence of mappings $R_k: X \times X \tto Y$, $k \in \N_0$, with closed graphs such that $(\bx,\bx) \in \inT \left(\bigcap_{k \in \N_0} \dom R_k \right)$ and  $0 \in R_k(\bx,\bx)$ for each $k \in \N_0$, and assume that there are positive constants $a$, $\gamma$, and $\ell$ satisfying
  \begin{equation} \label{11}
      \chi (\mathcal{H}(\bx)) + \ell + \gamma <  \inf\limits_{A \in \mathcal{H}(\bar{x})} \mbox{ \rm sur} \, G_A (\bx,0)
  \end{equation}
 such that
  \begin{equation} \label{eqgamma}
  \limsup_{x \to \bx, \ x \neq \bx} \frac{\sup_{k \in \N_0} \dist\big(0,R_k(x,\bx)\big)}{\|x - \bx\|} < \gamma,
 \end{equation}
and that, for all  $x$, $u$, $u' \in \ball(\bx,a)$ and all $k \in \N_0$, we have
 \begin{equation} \label{eqRk}
   R_k(x,u)\cap \ball(0,a) \subset R_k(x,u') + \ell \|u-u'\| \ball_Y.
 \end{equation}
Then  there exist $t \in (0,1)$  and $r>0$  such that,  for any starting point $x_0 \in \ball(\bx,r) $, there
exists a sequence $(x_k)$ in $\ball(\bx,r)$  generated by \eqref{niN} such that
\begin{equation} \label{t0}
 \| x_{k+1} - \bx \| \leq t  \| x_k - \bx \| \quad \mbox{for each} \quad k\in \N_0,
\end{equation}
that is, $(x_k)$ converges q-linearly to $\bx$.
\end{theorem}

\begin{proof} Shrink $a$, if necessary, to guarantee that
$$
\ball(\bar{x}, a)  \, \subset  \mbox{\rm dom} \, \mathcal{H}  \quad \mbox{and} \quad  \ball(\bar{x}, a) \times \ball(\bar{x}, a) \subset \dom R_k \quad \mbox{for all} \quad k\in \N_0.
$$
Let $c:= \chi (\mathcal{H}(\bx))$ and $m := \sup_{A \in \mathcal{H}(\bar{x})} \mbox{ \rm reg} \, G_A (\bx,0)$. By \eqref{11}, there are  $\mu > c$,  $\kappa > m$,  $\varepsilon >0$,
and $t \in (0,1)$ satisfying
\begin{equation} \label{eqmukappa}
 (\mu + \ell + \gamma + \varepsilon) \kappa < 1, \quad c +2 \varepsilon < \mu  \quad \mbox{and} \quad   {\kappa( \gamma   + \varepsilon)} < t(1-(\mu+\ell)\kappa).
\end{equation}

\medskip

{\sc Step 1.} {\it There exist $b\in (0,a)$ and $\theta \in (0, \kappa/(1-\mu\kappa))$ such that,
 for every $A \in \mathcal{H}(\ball(\bar{x},b))$  and
 for every  $(x,y)\in \ball(\bx,b) \times \ball(0,b)$, we have
 $$
    \dist\big(x, G_A^{-1}(y)\big) \leq \theta \dist\big(y, G_A(x)\big).
$$}
As $\mathcal{H}$ is upper semicontinuous at $\bx$, there is $\delta  \in (0,a)$ such that
\begin{equation} \label{ha}
 \mathcal{H} (x) \subset  \mathcal{H} (\bx) + \varepsilon \ball_{\mathcal{L}(X,Y)}
 \quad \mbox{for each} \ \  x \in \ball(\bx,\delta).
\end{equation}
 From the definition of measure of noncompactness, we find a finite subset $\mathcal{A}$ of $\mathcal{H}(\bar{x})$ such that
$$
 \mathcal{H} (\bx) \subset \mathcal{A} + \big(c + \varepsilon \big) \ball_{\mathcal{L}(X,Y)}.
$$
 Therefore, given  $x \in \ball(\bx, \delta)$, we have
$$
  \mathcal{H} (x) \overset{\eqref{ha}}{\subset}   \mathcal{A} + \big(c + \varepsilon \big) \ball_{\mathcal{L}(X,Y)} +  \varepsilon  \ball_{\mathcal{L}(X,Y)} = \mathcal{A} + (c + 2\varepsilon) \ball_{\mathcal{L}(X,Y)}.$$
The second inequality in \eqref{eqmukappa} implies that
\begin{equation} \label{Hy}
  \mathcal{H} (x)     \subset \mathcal{A}  +  \mu\ball_{\mathcal{L}(X,Y)}   \quad \mbox{for every} \quad   x \in \ball(\bx,\delta).
\end{equation}
Choose  $\theta$ to satisfy
$$m/{(1-\mu m)} < \theta < \kappa/(1-\mu\kappa),$$
and then choose
 $\tau \in (m, \kappa)$ with $\tau/(1-\mu\tau)< \theta$. Pick any $\bar{A} \in \mathcal{A}$ and
$A \in \mu\ball_{\mathcal{L}(X,Y)}$.  There exists ${\alpha} > 0$  such that
$$
 \dist\big(x, G_{\bar{A}}^{-1}(y)\big) \leq \tau \dist\big(y, G_{\bar{A}}(x)\big) \quad \mbox{for all} \quad (x,y) \in \ball(\bx, {\alpha}) \times \ball(0, {\alpha}).
$$
The mapping $ G_{\bar{A}}$ has
closed graph, because so does $F$.  Let $g(x):=  A(x - \bar{x})$, $x \in X$; then
$
 G_{\bar{A}+A} = G_{\bar{A}} + g.
$
Observe that  $g$ is single-valued, Lipschitz continuous with the constant $\mu$ such that $\mu \tau < 1$,  and $g(\bar{x}) = 0$.  We can apply \cite[Theorem~5G.3]{book} with $F:=G_{\bar{A}}$, $\bar{y} = 0$, $a=b:={\alpha}$, $\kappa:= \tau$, and $\kappa':=\theta$,
obtaining that there is ${\beta}=\beta({\bar{A}}) >0$, independent of $A$, such that the following {\tt claim} holds: {\it for
 each $y$, $y' \in \ball[0, {\beta}]$ and each $x \in \big(G_{\bar{A} + A}\big)^{-1}(y')\cap \ball[\bx,2\theta{\beta}]$, there is $ x'\in \big(G_{\bar{A}+A}\big)^{-1}(y)$ satisfying $  \|x - x'\| \leq \theta \|y - y'\|$.}

We show that, for each $(x,y) \in \ball(\bx,\theta{\beta}/3) \times  \ball(0, {\beta}/3)$,   we have
\begin{equation} \label{regcalG}
 \dist\big(x, \big(G_{\bar{A}+A}\big)^{-1}(y)\big) \leq \theta \dist(y,  G_{\bar{A}+A}(x)\big).
\end{equation}
To see this, fix any such
a pair $(x,y)$. Pick an arbitrary $y' \in G_{\bar{A}+A}(x)$ (if there is any).
If $\|y'\| \leq  {\beta}$, then the claim yields $x' \in\big(G_{\bar{A}+A}\big)^{-1}(y)$ with $\|x - x'\| \leq \theta \|y - y'\|$, and consequently,
\begin{eqnarray*}
  \dist\big(x, \big(G_{\bar{A}+A}\big)^{-1}(y)\big) &\leq & \| x - x' \| \leq \theta \, \|y - y'\|.
\end{eqnarray*}
On the other hand,  assuming that $\|y'\| > {\beta}$, we have $\|y' - y\| >  {\beta} - {\beta}/3 = 2 {\beta}/3$. Then, using the claim with $(x,y')$ replaced by $(\bx, 0)$, we find  $x' \in \big(G_{\bar{A}+A}\big)^{-1}(y)$ such that $\|\bar x - x'\| \leq \theta  \|y\|$. Consequently,
\begin{eqnarray*}
  \dist\big(x, \big(G_{\bar{A}+A}\big)^{-1}(y)\big)  &\leq & \| x - \bar x\| +   \dist\big(\bar x, \big(G_{\bar{A}+A}\big)^{-1}(y)\big) \leq \| x - \bar x\| + \|\bar x - x'\| \\
  & < &  \theta {\beta}/3 + \theta {\beta}/3  =  \theta (2{\beta}/3) <  \theta \|y-y'\|.
\end{eqnarray*}
Since $y' \in G_{\bar{A}+A}(x)$
is arbitrary,  \eqref{regcalG} is proved.

Summarizing,  given $\bar{A} \in \mathcal{A}$, there exists $ \beta:=\beta(\bar A) >0$ such that,  for each $A \in  \mu\ball_{\mathcal{L}(X,Y)}$  and each $(x,y) \in \ball(\bx,\theta{\beta}/3) \times  \ball(0, {\beta}/3)$,  inequality \eqref{regcalG} holds. Taking into account \eqref{Hy}, one has  $ \mathcal{H}(\ball(\bar{x},\delta)) \subset \mathcal{A}  +  \mu\ball_{\mathcal{L}(X,Y)}$.
Letting $b = \min_{\bar{A} \in \mathcal{A}} \{\delta, \beta(\bar{A})/3, \theta \beta(\bar{A})/3 \}$, we finish the proof of this step.

\medskip
{\sc Step 2.} {\it There exists $r > 0$ such that, for each $x \in \ball(\bx,r)$, each $A \in \mathcal{H}(x)$, and each $k \in \N_0$,  there is $x' \in \ball(\bx,r)$ such that
$$
 \big( f(x)+ A(x'-x)+F(x'){\big)} \cap R_k(x, x') \neq \emptyset \quad \mbox{and} \quad
  \|x' - \bar x\| \leq t \|x - \bar x\|.
$$}

Let $b$ and $\theta$ be the constants found in {\sc Step 1}. Using \eqref{eqH2} and \eqref{eqgamma}, we find a constant $\delta \in (0,  b/(1 + \gamma))$ such that, for every $ x \in \ball(\bx,\delta)\setminus\{\bx\}$ and every $k \in \N_0$, we have
\begin{equation} \label{aaa}
\sup_{A \in \mathcal{H}(x)}  \|f(x) - f(\bar{x}) - A (x - \bar{x})\| < \varepsilon {\|x - \bar{x}\|} \quad \mbox{and} \quad
 \dist\big(0,R_k(x,\bx)\big) <  \gamma {\|x - \bx\|}.
\end{equation}
The first inequality in \eqref{eqmukappa} implies that $\theta \ell < \kappa \ell/(1-\mu\kappa) < 1$. Let $r \in (0, \delta)$ be such that
$$
  r <  \frac{ \delta  (1- \theta \ell)}{ 2 (\varepsilon + \gamma) \max\{1, \theta\}} \, .
$$
Fix  an arbitrary $x \in \ball(\bx,r)$.   Choose  any $A \in \mathcal{H}(x)$ and
$k \in \N_0$.
If $x=\bx$, then, setting $x' := \bar{x}$, we are done because $0 \in R_k(\bx,\bx)$ and $0 \in f(\bx) + F(\bx)$. Assume that $x \neq \bx$. By \eqref{aaa} we find $\bar{z} \in - R_k(x, \bx)$ such that $\|\bar z\| < \gamma \|x-\bx\|$. Then
$$
 \ball(\bar z,\delta) \subset \ball(0, (1+\gamma)\delta) \subset \ball(0, b) \subset \ball(0,a).
$$
Consequently, for all  $u$, $u' \in \ball(\bx, \delta)$, we have
 \begin{eqnarray*}
   (- R_k(x,u))\cap \ball(\bar{z},\delta) & \subset  & (-R_k(x,u))\cap \ball(0,a) = - \big(R_k(x,u)\cap \ball(0,a)\big) \\
   &\overset{\eqref{eqRk}}{\subset}&  - \big(R_k(x,u') + \ell \|u-u'\| \ball_Y \big) = - R_k(x,u') + \ell \|u-u'\| \ball_Y.
  \end{eqnarray*}
 From {\sc Step 1} we get
 $$
    \dist\big(u, G_A^{-1}(v)\big) \leq \theta \dist\big(v, G_A(u)\big) \quad \mbox{for all} \quad (u,v)\in \ball(\bx,\delta) \times \ball(0,\delta).
 $$
 As $\theta \ell < 1$, applying Theorem~\ref{thmSumr} with $(F,G,\by, a, \kappa, \beta)$ replaced by $(G_A,-R_k(x,\cdot),0,\delta, \theta,(\varepsilon + \gamma) r)$, we get
 \begin{equation} \label{eqGARk}
    \dist\big(\bx, (G_A - R_k(x,\cdot))^{-1}(y)\big) \leq \frac{\theta}{1- \theta\ell} \|y - \bar{z}\| \quad \mbox{for all} \quad y \in \ball(\bar{z}, (\varepsilon + \gamma) r).
 \end{equation}
  Set
 \begin{equation}
   \label{y}
 y := f(\bar{x}) - f(x) +   A(x - \bar{x}).
\end{equation}
If $y = \bar{z}$, then $f(x) +A(\bar{x} - x) - f(\bar x) \in R_k(x,\bar x) \cap (f(x) + A(\bar x - x) + F(\bar x))$, and setting $x':=\bx$ we are done. Assume that $y \neq \bar{z}$.  The first inequality in \eqref{aaa} and the choice of $\bar{z}$ imply that
$$
0< \|y - \bar z \| \leq \|f(x) - f(\bar{x}) -  A (x - \bar{x})\|  + \|\bar z\| \\
 <  (\varepsilon  + \gamma)\|x - \bar{x}\| < (\varepsilon  + \gamma) r.
$$
Remembering that  $\theta< \kappa/ (1-\mu \kappa)$ and $\kappa \ell/(1-\mu \kappa) < 1$, and using the last inequality in \eqref{eqmukappa},  we get
$$
 \frac{\theta}{1-\theta\ell} <  \frac{\frac{\kappa}{1 - \mu \kappa}}{1-\frac{\kappa\ell}{1 - \mu \kappa}} =   \frac{\kappa}{1 -  (\mu +\ell)\kappa} < \frac{t}{\gamma + \varepsilon}.
$$
This and  \eqref{eqGARk} imply that there is  $x' \in (G_A - R_k(x,\cdot))^{-1}(y)$ such that
\begin{eqnarray*}
 \| x' - \bx \| & < &  \frac{t}{\gamma+\varepsilon} \|y - \bar z \| < \frac{t}{\varepsilon +\gamma} (\varepsilon +\gamma) \|x - \bar{x}\| = t  \|x - \bar{x}\|.
\end{eqnarray*}
Hence, $\|x' -\bx\|<r$ because  $t\in (0,1)$ and $x \in \ball(\bar x, r)$.  The choice of $y$ implies that
$$
 f(\bar{x}) -  f(x) + A(x - \bar{x}) \in G_{A}(x') - R_k(x,x')= f(\bx) + A(x' - \bx) + F(x') - R_k(x,x').
$$
Therefore  $0 \in f(x) + A(x' - x) + F(x') - R_k(x,x')$, which means that $\big( f(x)+ A(x'-x)+F(x'){\big)} \cap R_k(x, x') \neq \emptyset$. The proof of {\sc Step 2}  is finished.

\medskip

To conclude  the proof,  let $r>0$ be the constant found in {\sc Step 2}. Consider the iteration \eqref{niN}  and choose any $k \in \N_0$,
  $x_k \in \ball(\bx,r)$ and
$A_k \in {\cal H}(x_k)$.  Apply {\sc Step 2} with $A:=A_k$ and $x:=x_k$, and set $x_{k+1}:=x'$.
   Then $x_{k+1}$ satisfies  \eqref{niN} and \eqref{t0}.  It remains to choose any $x_0 \in  \ball(\bx,r)$ to
 obtain
this way  an infinite sequence $(x_k)$ in $\ball(\bx,r)$  generated by
\eqref{niN} and  satisfying (\ref{t0})  for all $k \in \N_0$.  Since $t \in (0,1)$,  $(x_k)$  converges linearly to $\bar{x}$.
\end{proof}

\begin{remark} \rm If \eqref{eqgamma} is replaced by a stronger condition
  $$ 
  \lim_{x \to \bx, \ x \neq \bx} \frac{\sup_{k \in \N_0} \dist\big(0,R_k(x,\bx)\big)}{\|x - \bx\|} = 0,
  $$ 
  then there is $r>0$  such that,  for any starting point $x_0 \in \ball(\bx,r) $, there
exists a sequence $(x_k)$ in $\ball(\bx,r)$  generated by \eqref{niN} such that  $(x_k)$ converges q-super-linearly to $\bx$, that is,  if there is $k_0 \in \N$ such that $x_k \neq \bx$ for all $k > k_0$ then
$ 
 \lim_{k \to \infty} \| x_{k+1} - \bx \|/\| x_k - \bx \| = 0
$. 
Indeed, in \eqref{eqmukappa} both the constants $\varepsilon$ and $\gamma$, and consequently, also $t$ can be chosen arbitrarily small.
\end{remark}

Suppose that $X:= \R^n$, $Y:=\R^m$, and $f$ is locally Lipschitz continuous. We can take, for example, Clarke's generalized Jacobian or Bouligand's limiting Jacobian as $\mathcal{H}$. Then $\mathcal{H}$ is upper semicontinuous and condition \eqref{eqH2} is satisfied when $f$ is semismooth at $\bx$ (with respect to the corresponding Jacobian). Moreover, $\chi(\mathcal{H}(\bx))=0$.
If, in addition,   $F \equiv 0$  and $R_k \equiv 0$ for each $k \in \N_0$, then the assumption
of regularity of all mappings $G_A$ in \eqref{G_A} is nothing else but
the requirement
that all matrices in $\mathcal{H}(\bx)$ have full-rank $m$, and we arrive at
the
classical result for semismooth Newton-type methods (see, for example, \cite{bonnans, Xu, ISbook, ACN, CDK, CDG,Uko}).

In \cite{ACN},  the following iterative process was studied:
{\it Choose a sequence  of set-valued mappings $A_k: X\times X\rightrightarrows Y$   and a starting point $x_0 \in X$, and generate a sequence $(x_k)$ in $X$ by taking $x_{k+1}$ to be a solution to the auxiliary inclusion}
\begin{equation}\label{Newton-Seq}
0\in A_k(x_{k+1},x_k)+F(x_{k+1}) \quad \mbox{for each} \quad \quad k \in \N_0.
\end{equation}
Theorem~4.1 therein  for iteration \eqref{Newton-Seq} is quite similar to Theorem~\ref{mrr2}
above with one important difference.
We assume that all the ``partial linearizations" $G_A$ in \eqref{G_A} are regular around $(\bx,0)$, while
in  \cite{ACN}
the mapping $f+F$ is assumed to be such.
Clearly, our assumption is weaker. Indeed take, for example,  $f(x):=|x|$, $x \in \R$, $F \equiv 0$, and $\mathcal{H}(x):=
x/|x|
$ if $x\neq 0$ and $\mathcal{H}(0) := \{-1,1\}$. Then $f$ is not even semiregular at $0$ while $\mathcal{H}$ satisfies all the assumptions in Theorem~\ref{mrr2}.


\begin{thebibliography}{00}

\bibitem{ACN} S. Adly, R. Cibulka, H. Van Ngai, Newton's method for solving inclusions using set-valued approximations,  SIAM J. Control Optim. 25 (2015) 159-184.

\bibitem{ADS2013}  M. Apetrii, M. Durea, R. Strugariu, On subregularity properties of set-valued mappings,
 Set-Valued Var. Anal. 21 (2013) 93--126.

\bibitem{AM2011} F.J. Arag\'{o}n Artacho,  B.S. Mordukhovich, Enhanced metric regularity and Lipschitzian properties of variational systems, J. Glob. Optim. 50 (2011) 145--167.

\bibitem{AF} J.P. Aubin, H. Frankowska, {Set-Valued Analysis}, Birkh\"auser, Boston,
1990.

\bibitem{BR}  W.-J. Beyn,  J. Rieger, An implicit function theorem for one-sided Lipschitz mappings,  Set-Valued Var. Anal. 19 (2011) 343--359.

\bibitem{bonnans}
   {J.F. Bonnans}, {Local analysis of Newton-type methods for variational inequalities and nonlinear programming},
    Appl. Math. Optim. 29 (1994) 161--186.

\bibitem{CDG} R. Cibulka, A.L. Dontchev, M.H. Geoffroy, Inexact Newton methods and Dennis--Mor\'e theorems for nonsmooth generalized equations,  SIAM J. Control Optim. 53 (2015) 1003-1019.

\bibitem{CDK} R. Cibulka, A.L. Dontchev, A.Y. Kruger, Strong metric subregularity of mappings in variational analysis and optimization, J. Math. Anal. Appl. 457 (2018) 1247--1282.




\bibitem{CF} R. Cibulka, M. Fabian, On primal regularity estimates for set-valued mappings, J. Math. Anal. Appl. 438 (2016) 444--464.

\bibitem{DmiMilOsm80} A.V. Dmitruk, A.A. Milyutin, N.P. Osmolovsky,
    Lyusternik's theorem and the theory of extrema,
    Russian Math. Surveys 35 (1980)
    11--51.


\bibitem{g_ad}   A.L. Dontchev, The Graves theorem revisited, J. Convex Anal. {3}   (1996)  45--53.

\bibitem{in} {A.L. Dontchev,  R.T. Rockafellar,} Convergence of inexact Newton methods for generalized equations, Math. Programming B { 139} (2013) 115--137.


\bibitem{book}  A.L. Dontchev,  R.T. Rockafellar, Implicit Functions and Solution Mappings, second ed., Springer, Dordrecht, 2014.

\bibitem{DS2012} M. Durea, R. Strugariu, Openness stability and implicit multifunction theorems: Applications to variational systems, Nonlinear Anal. 75 (2012) 1246--1259.


\bibitem{fp}  M. Fabian, D. Preiss,  {A   generalization   of   the   interior
mapping theorem of Clarke   and   Pourciau},   Comment.   Math.   Univ.
Carolinae  28 (1987) 311--324.


\bibitem{G} {L.M.} Graves, Some mapping theorems. Duke
Math. J.
{17} (1950) 111--114.

\bibitem{HG}   H. Hildebrand,  L.M. Graves, Implicit functions and their differentials in general
analysis. {Trans. AMS}  {29}  (1927) 127--153.

\bibitem{HNT2014} V.N. Huynh,  H.T. Nguyen, M. Th\'era,  Metric regularity of the sum of multifunctions and applications,
{ J. Optim. Theory Appl.} {160} (2014)  355--390.


\bibitem{IoffeSurvey} A.D. Ioffe, {Metric regularity--a survey. Part 1. Theory}, J. Aust. Math. Soc. 101 (2016) 188--243.


\bibitem{i2000} A.D. Ioffe, { Metric regularity and subdifferential calculus},
Russian math. Surveys  55 (2000) 501--558.

\bibitem{ISbook} A.F. Izmailov, M.V. Solodov, Newton-type Methods for Optimization and Variational Problems, Springer, New York, 2014.

\bibitem{JLbook} V. Jeyakumar, D.T. Luc, Nonsmooth Vector Functions and Continuous Optimization, Springer Science+Business Media, New York, 2008.

\bibitem{josephy}
   {N.H. Josephy}, {Newton's method for generalized equations},  Technical Summary Report no. 1965. Mathematics Research Center, University of Wisconsin, Madison, 1979.

\bibitem{KKbook} D. Klatte, B. Kummer, Nonsmooth Equations in Optimization. Regularity, Calculus, Methods and Applications, Kluwer Academic Publishers, Dordrecht, 2002.

\bibitem{Kru00} {A.Y. Kruger,}
Strict $(\varepsilon,\delta)$-semidifferentials and extremality of sets and functions, Dokl. Nats.
Akad. Nauk Belarusi 44 (2000) 21--24. In Russian.
\bibitem{Kru02} {A.Y. Kruger,}
Strict $(\varepsilon,\delta)$-semidifferentials and extremality conditions, Optimization 51  (2002)
539--554.
\bibitem{Kru04} {A.Y. Kruger,}
Weak stationarity: eliminating the gap between necessary and sufficient conditions,
Optimization 53  (2004) 147--164.
\bibitem{Kru05} {A.Y. Kruger,}
Stationarity and regularity of set systems, Pac. J. Optim. 1 (2005) 101--126.
\bibitem{Kru06} {A.Y. Kruger,}
About regularity of collections of sets, Set-Valued Anal. 14  (2006) 187--206.


\bibitem{AK2009} A.Y. Kruger, About stationarity and regularity in variational analysis, Taiwan. J. Math. 13 (2009) 1737--1785.

\bibitem{Kru15} {A.Y. Kruger,}
\newblock Error bounds and metric subregularity,
\newblock {Optimization} {64} (2015) 49--79.

\bibitem{KruTha15}
A.Y. Kruger, N.H. Thao,
\newblock Quantitative characterizations of regularity properties of collections of sets,
\newblock {J. Optim. Theory Appl. 164}  (2015) 41--67.


\bibitem{L}    L.A. Lyusternik, On the conditional extrema of functionals, Mat. Sbornik {41} (1934) 390--401  (in Russian).

\bibitem{JPP} J.-P. Penot, {Calculus without Derivatives}, Springer, New York, 2013.

\bibitem{Pourciau} B.H. Pourciau, Analysis  and  optimization  of Lipschitz ­continuous mappings,  J. Opt. Theory  Appl.  { 22} (1977) 311--351.

\bibitem{rob}   {S.M. Robinson}, {Strongly regular generalized equations}, Math. Oper. Res. 5 (1980) 43--62.

\bibitem{u} A. Uderzo,  A strong metric subregularity analysis of nonsmooth mappings via steepest displacement rate, { J. Optim. Theory Appl.} {171} (2016)  573--599.

\bibitem{Ude4} A. Uderzo,
    An implicit multifunction theorem for the
        hemiregularity of mappings with application to
        constrained optimization,
    {Preprint, arXiv:1703.10552} (2017)
    {1--16}.

\bibitem{Uko} L.U. Uko, Generalized equations and the generalized Newton method, Mathematical programming 73 (1996) 251-268.

\bibitem{Xu}  {H. Xu}, {Set-valued approximations and Newton's methods}, Math. Program. Ser. A 84 (1999) 401--420.

\end{thebibliography}
\end{document}